\documentclass[11pt, a4paper,leqno]{amsart}

\usepackage{amsmath,amsthm,amscd,amssymb,amsfonts, amsbsy}
\usepackage{latexsym}
\usepackage{txfonts}
\usepackage{exscale}
\usepackage{enumitem}

\usepackage[colorlinks=true, pdfstartview=FitV, linkcolor=blue, citecolor=red, urlcolor=blue]{hyperref}

\usepackage{latexsym}
\usepackage{tikz}
\usepackage{tkz-euclide}
\usetkzobj{all}
\usetikzlibrary{intersections}
\usetikzlibrary{shapes,arrows,calc}

\usepackage{pgf}
\usepackage{color}

\calclayout
\allowdisplaybreaks

\theoremstyle{plain}
\newtheorem{theorem}[equation]{Theorem}
\newtheorem{lemma}[equation]{Lemma}

\newtheorem{proposition}[equation]{Proposition}

\theoremstyle{definition}
\newtheorem{definition}[equation]{Definition}

\theoremstyle{remark}
\newtheorem{remark}[equation]{Remark}

\numberwithin{equation}{section}
\usepackage{chngcntr} %% needed for counterwithin
\counterwithin{figure}{section}

\newcommand{\dist}{\operatorname{dist}}

\newcommand{\re}{\mathbb{R}}

\newcommand{\ree}{\mathbb{R}^{n+1}}
\newcommand{\N}{\mathbb{N}}
\newcommand{\Z}{\mathbb{Z}}
\newcommand{\dd}{\mathbb{D}}

\newcommand{\om}{\Omega}

\newcommand{\F}{\mathcal{F}}

\newcommand{\W}{\mathcal{W}}

\newcommand{\mut}{\mathfrak{m}}

\newcommand{\pom}{\partial\Omega}

\newcommand{\hm}{\omega}

\renewcommand{\P}{\mathcal{P}}

\renewcommand{\emptyset}{\mbox{\textup{\O}}}
\newcommand{\tinyemptyset}{\mbox{\tiny \textup{\O}}}

\DeclareMathOperator{\diam}{diam}

\DeclareMathOperator{\interior}{int}

\def\div{\mathop{\operatorname{div}}\nolimits}

\newcommand\rh{\mathcal{H}}
\newcommand\rd{d}
\newcommand{\mc}[1]{\mathcal{#1}}
\begin{document}
\allowdisplaybreaks

\title[Rectifiability and elliptic measures]{Rectifiability and elliptic measures on 1-sided NTA domains with Ahlfors-David regular boundaries}

\author{Murat Akman}

\address{Murat Akman\\
Mathematical Sciences Research Institute
\\
17 Gauss Way
\\
Berkeley, CA 94720, USA}
\email{makman@msri.org}

\author{Matthew Badger}

\address{Matthew Badger\\ Department of Mathematics
\\
University of Connecticut
\\
Storrs,
CT 06269-3009, USA}
\email{matthew.badger@uconn.edu}

\author{Steve Hofmann}

\address{Steve Hofmann
\\
Department of Mathematics
\\
University of Missouri
\\
Columbia, MO 65211, USA} \email{hofmanns@missouri.edu}

\author{Jos\'e Mar{\'\i}a Martell}

\address{Jos\'e Mar{\'\i}a Martell\\
Instituto de Ciencias Matem\'aticas CSIC-UAM-UC3M-UCM\\
Consejo Superior de Investigaciones Cient{\'\i}ficas\\
C/ Nicol\'as Cabrera, 13-15\\
E-28049 Madrid, Spain} \email{chema.martell@icmat.es}

\thanks{The first and last authors have been supported in part by the Spanish Ministry of Economy and Competitiveness, through the ``Severo Ochoa Programme for Centres of Excellence in R\&D'' (SEV-2015-0554) and they acknowledge that the research leading to these results has received funding from the European Research Council under the European Union's Seventh Framework Programme (FP7/2007-2013)/ ERC agreement no. 615112 HAPDEGMT. The second  author was partially supported by an NSF postdoctoral fellowship, DMS 1203497, and by NSF grant DMS 1500382. The third author was partially supported by NSF grant DMS 1361701.}

\date{July 6, 2015. \textit{Revised}\textup{:} \today}

\subjclass[2010]{28A75, 28A78, 31A15, 31B05, 35J25, 42B37, 49Q15}

\keywords{NTA domains, 1-sided NTA domains, uniform domains, Ahlfors-David regular sets, rectifiability, harmonic measure, elliptic measure, surface measure, linearly approximability, elliptic operators}

\begin{abstract}
Let $\Omega \subset \ree$, $n\geq 2$, be 1-sided NTA domain also known as uniform domain), i.e., a domain which satisfies interior Corkscrew
and Harnack Chain conditions, and assume that $\pom$ is $n$-dimensional Ahlfors-David regular. We characterize the rectifiability of $\pom$ in terms of the absolute continuity of surface measure with respect to harmonic measure. We also show that these are equivalent to the fact that $\pom$ can be covered $\mathcal{H}^n$-a.e.~by a countable union of portions of boundaries of bounded chord-arc subdomains of $\Omega$ and to the fact that $\partial\Omega$ possesses exterior corkscrew points in a qualitative way $\mathcal{H}^n$-a.e. Our methods apply to harmonic measure and also to elliptic measures associated with real symmetric second order divergence form elliptic operators with locally Lipschitz coefficients whose derivatives  satisfy a natural qualitative Carleson condition.
\end{abstract}

\maketitle

\tableofcontents

\section{Introduction and statement of main results}\label{s1}
A well known result of F. and M. Riesz says that if $\Omega$ is a
simply connected planar domain whose boundary is a Jordan curve of finite length, then harmonic measure $\omega$ and arclength $\rh^1|_{\partial\Omega}$ are mutually absolutely continuous. A quantitative version of this theorem was proved by Lavrentiev in \cite{Lav}. Due to examples of Bishop and Jones in \cite{BiJo} in the planar case, and of Ziemer in \cite{Z} and Wu in \cite{Wu} in higher dimensions, neither $\rh^{n}|_{\partial\Omega}\ll\omega$ nor $\omega\ll\rh^{n}$ are true for arbitrary simply connected domains $\Omega\subset\ree$ with $\rh^n(\partial\Omega)<\infty$ without imposing additional topological and/or non-topological conditions on $\partial\Omega$. Quantitative mutual absolute continuity of harmonic measure and surface measure in higher dimensions was proven when $\Omega$ is a Lipschitz domain by Dahlberg in \cite{Dah}, and when $\Omega$ is  non-tangentially accessible (NTA) (see Definition \ref{def-nta}) and $\partial\Omega$ is Ahlfors-David regular (ADR, see Definition \ref{def-ADR}) independently by David and Jerison in \cite{DJ} and by Semmes in \cite{Sem}. It is now known that if $\Omega$ is a 1-sided NTA domain (see Definition \ref{def-1nta}) with ADR boundary, then the following are equivalent:
\begin{align}
\label{AHMNT}
\begin{split}
&\mbox{(i)\, $\partial\Omega$ is Uniformly Rectifiable},\\
&\mbox{(ii)\, $\Omega$ is an NTA domain (and therefore $\Omega$ is a chord-arc domain)},\\
&\mbox{(iii)\, $\omega\in A_{\infty}$},\\
&\mbox{(iv)\, $\omega\in\mbox{ weak}-A_{\infty}$}.
\end{split}
\end{align}
Here (iii) and (iv) should be understood in a scale invariant sense (see Definition \ref{defi-Ainfty}). 
The implication (i) implies (ii) was proved in \cite{AHMNT}; (ii) implies (iii) was proved independently in \cite{DJ, Sem} as
mentioned above; (iii) implies (iv) is trivial; and (iv) implies (i) was proved in \cite{HMU}. On the other hand, in \cite{Ba}, it was shown that if $\Omega$ is an NTA domain with $\rh^{n}(\partial\Omega)<\infty$, then $\partial\Omega$ is $n$-rectifiable and $\rh^{n}|_{\partial\Omega}\ll\omega$. Moreover, it was also shown in \cite{Ba} that if $\Omega$ is an NTA domain, then $\omega\ll \rh^n\ll \omega$ on $A$, where
\[
A:=\left\{x\in\partial\Omega:\, \liminf\limits_{r\to 0} \frac{\rh^{n}(\partial\Omega\cap B(x,r))}{r^{n}}<\infty\right\}.
\]
However, due to an example of Azzam, Mourgoglou, and Tolsa in \cite{AMT}, harmonic measure is not necessarily absolutely continuous with respect to surface measure on the entire boundary of an NTA domain of locally finite perimeter. In particular, the authors of \cite{AMT} constructed Reifenberg flat domains $\Omega$ with locally finite surface measure $\rh^n|_{\partial\Omega}$ and Borel sets $E\subset\partial\Omega$ with $\omega(E)>0=\rh^{n}(E)$. (In fact, the sets $E$ have Hausdorff dimension less than $n$.) Therefore, in order to ensure that $\omega\ll \rh^{n}$ on the full boundary of an NTA domain of locally finite perimeter, one needs to identify some additional qualitative or quantitative conditions on $\partial\Omega$. For related results on $p$-harmonic measure, see \cite{LN11}.

The main result proved in this article is the following qualitative version of (\ref{AHMNT}) (see Section \ref{ssdefs} for the precise definitions).
\begin{theorem}
\label{main}
Let $\Omega\subset\ree$, $n\ge 2$, be a 1-sided NTA domain whose boundary is ADR. Write $\omega:=\omega^{X_0}$ for the harmonic measure of $\Omega$ with pole at $X_0$, any given point in $\Omega$, and write $\sigma:=\rh^n|_{\partial\Omega}$ for surface measure on $\partial\Omega$. Then the following statements are equivalent:
\begin{enumerate}[label=\textup{(\alph*)},ref=\alph*]\itemsep=0.1cm
\item \label{maina} $\partial\Omega$ is rectifiable;

\item \label{mainb} there exists a set $F\subset\partial\Omega$ and a constant $c_{0}$, $0<c_0<1$  such that $\sigma(F)=0$ and for all $x\in \partial\Omega\setminus F$ there is  $r_{x}>0$  for which $\Delta(x,r)=\partial\Omega\cap B(x,r)$ has an exterior corkscrew point (that is, a corkscrew point with respect to the open set $\Omega_{\rm ext}=\ree\setminus\overline{\Omega}$) for all $0<r<r_{x}$ with implicit constant $c_0$;

\item \label{mainc} $\sigma\ll \omega$ on $\partial\Omega$;

\item \label{maind} $\partial\Omega=F_0\cup\left(\bigcup_N F_{N}\right)$, where $\sigma(F_{0})=0$ and $F_{N}=\partial\Omega\cap \partial\Omega_{N}$ for some bounded chord-arc domain  $\Omega_{N}\subset\Omega$;

\item \label{maine} $\partial\Omega=F_0\cup\left(\bigcup_N F_{N}\right)$, where $\sigma(F_{0})=0$ and for each $N$ there exist constants $\theta_N$, $\theta_N'>0$ and $C_N>1$ such that
$$
C_N^{-1}\sigma(F)^{\theta_N'}
\le
\omega(F)
\le
C_N\,\sigma(F)^{\theta_N}
\qquad
\forall\,F\subset F_N.
$$
\end{enumerate}

\end{theorem}

The proof of Theorem \ref{main} is in Section \ref{sect:proof-main} and goes as follows. First, observe that  \eqref{maine} easily gives  \eqref{mainc}. Second, \eqref{maind} yields \eqref{maina}, because the boundary of any chord arc domain is rectifiable (e.g., see \cite{DJ} or \cite{Ba}).  In Section \ref{rectandapprox}, we use a notion of approximate tangent planes from geometric measure theory to show that \eqref{maina} implies \eqref{mainb}. Next, we prove in Section \ref{bimpliesd} that \eqref{mainb} implies \eqref{maind} by constructing certain sawtooth domains $\Omega_{\mc{F}, Q_{0}}$, which are bounded chord-arc subdomains of $\Omega$. In Section \ref{dimpliesc}, we verify \eqref{maind} implies \eqref{mainc} by a straightforward use of the maximum principle. In Section \ref{cimpliesd}, we first show that some family of bad cubes (for which the exterior corkscrew condition fails) satisfies a Carleson packing condition. From there, we obtain that another suitable family of sawtooth domains $\Omega_{\mc{F},Q_{0}}$ are chord-arc domains and show that \eqref{mainc} implies \eqref{maind}.  To complete the proof, in Section \ref{bimpliese} we demonstrate that  \eqref{mainb} implies  \eqref{maine} by using a variant of
the Dahlberg-Jerison-Kenig sawtooth lemma and a certain projection operator.

Although our main result is written in terms of harmonic measure, our methods allow for more general elliptic measures. In particular, in Theorem \ref{main} we can replace harmonic measure $\omega$ with elliptic measures $\omega_L$ corresponding to a class of divergence form elliptic operators whose coefficients are locally Lipschitz and obey a natural Carleson measure condition. Our class of operators is motivated by the results in \cite{KP} and the recent work \cite{HMT}. The operators considered in \cite{KP} have the property that they are good (i.e., their elliptic measure is $A_\infty$) in chord-arc subdomains. This is relevant in the proof of \eqref{mainb} implies \eqref{maine}, where such a property is used for the Laplacian. On the other hand, \cite{HMT} contains
 a generalized version of the implication (iv) $\implies$ (ii) in \eqref{AHMNT},
valid for a class of elliptic operators.
In \cite{HMT}, there is an ``integration by parts'' argument that
allows the authors to obtain localized square functions estimates and we use
a similar argument here in the proof of \eqref{mainc} implies
\eqref{maind}. Here we shall assume qualitative versions of the conditions in \cite{KP}, \cite{HMT} that allow us to follow their ideas in a qualitative way. The precise result is as follows:

\begin{theorem}\label{mainvarc}
Let $\Omega\subset\ree$, $n\ge 2$, be a 1-sided NTA domain whose boundary is ADR. Let $Lu:=-\div (A\,\nabla u)$ and assume that $A$ is uniformly elliptic, real, symmetric, $A\in {\rm Lip}_{\rm loc}(\Omega)$, and for every ball $B=B(x,R)$ with $x\in\pom$ and $0<R<\diam(\pom)$, there exists $C_{B}$ such that
\begin{equation}\label{main-A-Car}
\sup_{\substack{y\in B\cap \partial\Omega \\ 0<r\le R}} \frac{1}{r^n}\iint_{B(y,r)\cap\Omega} \Big(\sup_{Z\in B(X,\delta(X)/2)}|\nabla A(Z)|\Big)\,dX\le C_B,
%,\qquad
%\mbox{where}
%\quad
%a(X)=\sup_{Z\in B(X,\delta(X)/2)}|\nabla A(Z)|
\end{equation}
 where $\delta(X)=\dist(X,\partial\Omega)$.
Write $\omega_{L}:=\omega_{L}^{X_0}$ for the elliptic measure of $\Omega$ associated to $L$ with pole at $X_0$, any given point in $\Omega$,  and write $\sigma:=\rh^n|_{\partial\Omega}$ for surface measure on $\partial\Omega$. Then the equivalent statements \eqref{maina}--\eqref{maine} in Theorem \ref{main} are also equivalent to the following statements:

\begin{enumerate}[label=\textup{(\alph*')},ref=\alph*']\itemsep=0.1cm

\setcounter{enumi}{2} \item \label{maincvc} $\sigma\ll \omega_L$ on $\partial\Omega$;

\setcounter{enumi}{4}\item \label{mainevc} $\partial\Omega=F_0\cup\left(\bigcup_N F_{N}\right)$, where $\sigma(F_{0})=0$ and for each $N$ there exist constants $\theta_N$, $\theta_N'>0$ and $C_N>1$ such that
$$
C_N^{-1}\sigma(F)^{\theta_N'}
\le
\omega_L(F)
\le
C_N\,\sigma(F)^{\theta_N}
\qquad
\forall\,F\subset F_N.
$$
\end{enumerate}

\end{theorem}

By an easy compactness argument, to invoke Theorem \ref{mainvarc} it is enough to verify that \eqref{main-A-Car} holds on balls $B=B(x,R_x)$ for every $x\in\pom$, for some $R_x>0$ depending on $x$. Examples of operators $Lu:=-\div (A\,\nabla u)$ where this result applies include the case of coefficients  $A$ which are locally Lipschitz
in $\Omega$,  with $|\nabla A|\in L^\infty(B(x,r_x)\cap\Omega)$ for every $x\in\pom$, for some $r_x>0$. More generally, one may assume that there is $\epsilon>0$ such that $|\nabla A(X)|\,\delta(X)^{1-\epsilon}\to 0$ as $X\to x$ along $X\in\Omega$ for every $x\in\partial\Omega$.

The proof of Theorem \ref{mainvarc} is given in Section \ref{variablecase}. Note that \eqref{mainevc} easily implies \eqref{maincvc}. To complete the proof, we show that \eqref{maincvc} implies \eqref{maind} in Section \ref{ss:cprime-d} and \eqref{mainb} implies \eqref{mainevc} in Section \ref{ss:b-cprime}.

Finally, in Section \ref{section:example},  we construct an example of a domain $\Omega_\star$ satisfying the required background hypotheses (i.e., 1-sided NTA with ADR boundary) for which \eqref{maina}--\eqref{maine} in Theorem \ref{main} hold, but (i)--(iv) in \eqref{AHMNT} fail. In particular, for this example, harmonic measure (and the elliptic measures in Theorem \ref{mainvarc}) belongs to neither $A_\infty$ nor to weak-$A_\infty$, but nevertheless satisfies the weaker absolutely continuity conditions \eqref{mainc}, \eqref{maine} (and \eqref{maincvc}, \eqref{mainevc}).

We note that some interesting related work has recently appeared, or been carried out, while this manuscript was in preparation, due to  Mourgoglou
\cite{Mo}, Azzam, Mourgoglou and Tolsa \cite{AMT2}, and Mayboroda, Tolsa, Volberg and the two last authors of the present
paper \cite{HMMTV} which sharpens our results
in the special case of Laplace's equation.  In the first manuscript,
the author obtains the implication \eqref{maina} implies \eqref{mainc} of our Theorem \ref{main}, but with the upper ADR bound on $\partial\Omega$
replaced by the weaker qualitative condition that $\mathcal{H}^n|_{\pom}$ is
locally finite. Moreover,
in \cite{AMT2}, the authors obtain the converse direction
\eqref{mainc} implies \eqref{maina} (as well as results concerning rectifiability of harmonic measure,
provided that $\hm\ll \sigma$),
replacing the ADR hypothesis
by the weaker qualitative assumption that
$\mathcal{H}^n|_{\pom}$ is positive and locally finite, and assuming
only a  ``porosity"
(i.e., Corkscrew) condition in the complement of $\pom$, in lieu of the stronger
1-sided NTA  assumption.  In \cite{HMMTV} the same result is proved removing the porosity assumption.
Both \cite{AMT2} and the follow-up version \cite{HMMTV}, merged into the paper \cite{AHMMMTV},
 rely on recent deep results
of \cite{NToV}, \cite{NToV2}, concerning connections between rectifiability and the behavior of
Riesz transforms.    The use of these Riesz transform results allows for non-trivial
weakening of the hypotheses as described above,  but on the other hand, the Riesz transforms are tied explicitly to harmonicity.
Our methods in the present paper, involving localized square function estimates,
seem to require a stronger connectivity hypothesis
(i.e., the 1-sided NTA, also known as ``uniform domain", assumption), but are more robust in the sense that
they allow treatment of variable coefficient operators.

A related result for harmonic measure and $p$-harmonic measure has been also obtained  by the last  two authors of the present paper in collaboration with Le and Nystr\"om \cite{HLMN} (see \cite{HM-Ainfty-UR} for a version just containing the harmonic case): if $\Omega$ is an open set with ADR boundary and harmonic measure satisfies a weak-$A_\infty$ condition on $\pom$ then $\pom$ is Uniformly Rectifiable. This corresponds to a quantitative version of the implication \eqref{mainc} implies \eqref{maina} of our Theorem \ref{main} in a setting without connectivity assumptions. The converse of this result, that is, that the complement of a Uniformly Rectifiable set has ``interior big pieces of good harmonic measure estimates'', has been recently proved by Bortz and the third author of this paper \cite{BoH}. This can be seen as a quantitative version of \eqref{mainc} (or \eqref{maine}) implies \eqref{maina} in Theorem \ref{main}.

\subsection{Notation and conventions}
\label{notconv}
\begin{list}{$\bullet$}{\leftmargin=0.4cm  \itemsep=0.2cm}

\item We use the letters $c,C$ to denote harmless positive constants, not necessarily the same at each occurrence, which depend on at most dimension and the
constants appearing in the hypotheses of the theorems (that is, on ``allowable parameters''). Unless otherwise specified, upper case constants are greater than $1$  and lower case constants are smaller than $1$. We write $a\lesssim b$ or $a \approx b$ to denote $a \leq C b$ or $0< c \leq a/b\leq C$ for some constants $c$ and $C$ following the convention above, respectively.

\item Given a domain $\Omega \subset \ree$, we shall
use lower case letters $x,y,z$, etc.~ to denote points on $\partial \Omega$, and capital letters
$X,Y,Z$, etc.~to denote generic points in $\ree$ (especially those in $\ree\setminus \partial\Omega$).

\item The open $(n+1)$-dimensional Euclidean ball of radius $r$ will be denoted
$B(x,r)$ when the center $x$ lies on $\partial \Omega$, and denoted $B(X,r)$ when the center
$X \in \ree\setminus \partial\Omega$.  A {\it surface ball} is denoted
$\Delta(x,r):= B(x,r) \cap\partial\Omega.$

\item If $\pom$ is bounded, it is always understood (unless otherwise specified) that all surface balls have radii controlled by the diameter of $\pom$: that is, if $\Delta=\Delta(x,r)$, then $r\lesssim \diam(\pom)$. Note that in this way $\Delta=\pom$ if $\diam(\pom)<r\lesssim \diam(\pom)$.

%\item Given a Euclidean ball $B$ or surface ball $\Delta$, its radius will be denoted
%$r_B$ or $r_\Delta$, respectively.

%\item Given a Euclidean or surface ball $B= B(X,r)$ or $\Delta = \Delta(x,r)$, its concentric
%dilate by a factor of $\kappa >0$ will be denoted
%by $\kappa B := B(X,\kappa r)$ or $\kappa \Delta := \Delta(x,\kappa r).$

\item Let $\dist(A,B):=\inf_{a\in A}\inf_{b\in B}|a-b|$ denote the usual Euclidean distance between sets $A$ and $B$. For $X \in \ree$, let $\delta(X):= \dist(X,\partial\Omega)$.

\item Let $\rh^n$ denote $n$-dimensional Hausdorff measure and let $\sigma := \rh^n\big|_{\partial\Omega}$ denote the surface measure on $\partial \Omega$.

%\item For a Borel set $A\subset \ree$, we let $1_A$ denote the usual
%indicator function of $A$, i.e. $1_A(x) = 1$ if $x\in A$, and $1_A(x)= 0$ if $x\notin A$.

\item For a generic set $A\subset \ree$, we let $\interior(A)$ denote the interior of $A$. However, when $A\subset \partial\Omega$, we let $\interior(A)$ denote the interior of $A$ relative to $\partial\Omega$; that is, $\interior(A)$ is the largest relatively open set in $\partial\Omega$ contained in $A$. In addition, for $A\subset \partial\Omega$, we define
the boundary $\partial A := \overline{A} \setminus {\rm int}(A)$ using our convention on $\interior(A)$.

%\item For a Borel set $A$, we denote by $\mathcal{C}(A)$ the space of continuous functions on
%$A$, by $\mathcal{C}_0(A)$ the subspace of $\mathcal{C}(A)$
%with compact support in $A$.

%\item For a Borel subset $A\subset\partial\Omega$, we
%set $\fint_A f d\sigma := \sigma(A)^{-1} \int_A f d\sigma$.

\item We shall use the letter $I$ (and sometimes $J$)
to denote a closed $(n+1)$-dimensional Euclidean cube with sides
parallel to the co-ordinate axes, and we let $\ell(I)$ denote the side length of $I$.

\item We use $Q$ to denote a dyadic ``cube''
on $\partial \Omega$, which exists whenever $\partial \Omega$ is ADR  (see \cite{Ch,DS1}) and enjoy certain properties
enumerated in Lemma \ref{lemma:Christ} below.
\end{list}

\subsection{Some definitions}
\label{ssdefs} %and weak-$A_\infty$}

\begin{definition}[\bf Ahlfors-David regular]
\label{def-ADR}
 We say that a closed set $E \subset \ree$ is $n$-dimensional ADR (or simply ADR) if
there is some uniform constant $C$ such that
\begin{align*}
\frac1C\, r^n \leq \rh^{n}(E\cap B(x,r)) \leq C\, r^n \quad \forall\, r\in(0,\diam(E)),\, x \in E.
\end{align*}
\end{definition}
Following \cite{JK}, we state the definition of Corkscrew condition, Harnack Chain condition, and NTA domains.

\begin{definition}[\bf Corkscrew condition]
%\label{def1.cork}
We say that an open set $\Omega\subset \ree$
satisfies the (interior) {\it Corkscrew condition} if for some uniform constant $c$, $0<c<1$, and
for every surface ball $\Delta:=\Delta(x,r),$ with $x\in \partial\Omega$ and
$0<r<\diam(\partial\Omega)$, there is a ball
$B(X_\Delta,cr)\subset B(x,r)\cap\Omega$.  The point $X_\Delta\subset \Omega$ is called
a (interior) {\it corkscrew point relative to} $\Delta,$ (or, relative to $B$). We note that  we may allow
$r<C\diam(\pom)$ for any fixed $C$, simply by adjusting the constant $c$.

Analogously, we say that an open set $\Omega\subset \ree$ satisfies the exterior {\it Corkscrew condition}
if the open set $\Omega_{\rm ext}=\ree\setminus\overline{\Omega}$ satisfies the (interior) Corkscrew condition. Also, if we say that $X_\Delta$ is  an {\it exterior corkscrew point relative to} $\Delta,$ (or, relative to $B$), we mean that $X_\Delta\in \Omega_{\rm ext}$ is  an (interior) {\it corkscrew point relative to} $\Delta,$ (or, relative to $B$) for the open set $\Omega_{\rm ext}$. 
\end{definition}

\begin{definition}[\bf Harnack Chain condition]
%\label{def1.hc}
We say that $\Omega$ satisfies the {\it Harnack Chain condition} if there is a uniform constant $C$ such that
for every $\rho >0,\, \Lambda\geq 1$, and every pair of points
$X,X' \in \Omega$ with $\delta(X),\,\delta(X') \geq\rho$ and $|X-X'|<\Lambda\,\rho$, there is a chain of
open balls
$B_1,\dots,B_N \subset \Omega$, $N\leq C(\Lambda)$,
with $X\in B_1,\, X'\in B_N,$ $B_k\cap B_{k+1}\neq \emptyset$
and $C^{-1}\diam (B_k) \leq \dist (B_k,\partial\Omega)\leq C\diam (B_k).$  The chain of balls is called
a {\it Harnack Chain}.
\end{definition}

\begin{definition}[\bf 1-sided NTA domain]
\label{def-1nta}
If $\Omega$ satisfies both the Corkscrew and Harnack Chain conditions, then we say that
$\Omega$ is a {\it 1-sided NTA domain}.
\end{definition}
\begin{definition}[\bf NTA domain]
\label{def-nta}
We say that a domain $\Omega$ is an {\it NTA  domain} if
it is a 1-sided NTA domain and if, in addition, $\om_{\rm ext}:= \ree\setminus \overline{\Omega}$
also satisfies the Corkscrew condition.
\end{definition}
\begin{remark} The abbreviation NTA stands for non-tangentially accessible. In the literature, 1-sided NTA domains are also called \textit{uniform domains}. We remark that the 1-sided NTA condition is a quantitative form of path connectedness.
\end{remark}

\begin{definition}[\bf Chord-arc domain]
A \emph{chord-arc domain} $\Omega$ is an NTA domain with ADR boundary.
\end{definition}

We next give definition of rectifiability. For general background, see \cite{Ma95}.
\begin{definition}[{\bf Rectifiability}]
\label{rectifiable}
A set in $E\subset \mathbb{R}^{n+1}$ is called $n$-rectifiable if there exist Lipschitz maps $f_{i}:\mathbb{R}^{n}\to \mathbb{R}^{n+1}$, $i=1,2,\ldots,$ such that
\[
\rh^{n}\left(E\setminus \bigcup\limits_{i=1}^{\infty}f_{i}(\mathbb{R}^{n})\right)=0.
\]
\end{definition}

\begin{definition}
[{\bf $A_\infty$, weak-$A_\infty$}]\label{defi-Ainfty}
Given $\Omega\subset\ree$, a 1-sided NTA domain with ADR boundary, let $\omega$ be the associated harmonic measure (or some other elliptic measure). We say that $\omega$ is weak-$A_\infty$ if there exist positive constants $C$ and $\theta$ (depending on $n$ and the 1-sided NTA and ADR constants) such that for every surface ball $\Delta_0=B_0\cap \pom$, with $B_0$ centered at $\pom$ and radius smaller than $\diam(\pom)$, and for every surface ball $\Delta=B\cap \pom$, with $B$ centered at $\pom$ and $2\, B \subset B_0$,
\begin{equation}\label{eq1.wainfty}
\omega^{X_{\Delta_0}} (F) \leq C \left(\frac{\sigma(F)}{\sigma(\Delta)}\right)^\theta\,\omega^{X_{\Delta_0}} (2\,\Delta),
\qquad \forall\,F\subset \Delta.
\end{equation}
Analogously, we say that $\omega\in A_\infty$ if the previous condition holds with $B\subset B_0$, in place of $2\, B \subset B_0$, and if we can write  $\omega^{X_{\Delta_0}} (\Delta)$, in place of $\omega^{X_{\Delta_0}} (2\,\Delta)$, in the right hand side of \eqref{eq1.wainfty}
\end{definition}

\subsection{Dyadic grids and sawtooths}
\label{ss:grid}
In this subsection we give a lemma concerning the existence of ``dyadic grid'' which can be found in \cite{DS1,DS2,Ch}.

\begin{lemma}[\bf Existence and properties of the ``dyadic grid'']
\label{lemma:Christ}
If $E\subset \ree$ is ADR, then there exist
constants $ a_0>0$, $\eta>0$, and $C_1<\infty$, depending only on dimension and the
ADR constant, and for each $k \in \mathbb{Z}$
there exists a collection of Borel sets (``cubes'')
$$
\mathbb{D}_k:=\{Q_{j}^k\subset E: j\in \mathfrak{I}_k\},$$ where
$\mathfrak{I}_k$ denotes some (possibly finite) index set depending on $k$, satisfying the following properties.

\begin{list}{$(\theenumi)$}{\usecounter{enumi}\leftmargin=.8cm
\labelwidth=.8cm\itemsep=0.2cm\topsep=.1cm
\renewcommand{\theenumi}{\roman{enumi}}}

\item $E=\bigcup_{j}Q_{j}^k\,\,$ for each
$k\in{\mathbb Z}$.

\item If $m\geq k$ then either $Q_{i}^{m}\subset Q_{j}^{k}$ or
$Q_{i}^{m}\cap Q_{j}^{k}=\emptyset$.

\item For each $(j,k)$ and each $m<k$, there is a unique
$i$ such that $Q_{j}^k\subset Q_{i}^m$.

\item The diameter of each $Q_{j}^k$ is at most $C_12^{-k}$.

\item Each $Q_{j}^k$ contains some surface ball $\Delta \big(x^k_{j},a_02^{-k}\big):=
B\big(x^k_{j},a_02^{-k}\big)\cap E$.

\item $\rh^{n}\left(\left\{x\in Q^k_j:{\rm dist}(x,E\setminus Q^k_j)\leq \tau \,2^{-k}\right\}\right)\leq
C_1\,\tau^\eta\,\rh^{n}\left(Q^k_j\right)$ for all $k$ and $j$ and for all $\tau\in (0,a_0)$.
\end{list}
\end{lemma}

Some notations and remarks are in order concerning this lemma.

\begin{list}{$\bullet$}{\leftmargin=0.4cm  \itemsep=0.2cm}

\item In the setting of a general space of homogeneous type, this lemma has been proved by Christ
\cite{Ch}, with the
dyadic parameter $1/2$ replaced by some constant $\delta \in (0,1)$.
In fact, one may always take $\delta = 1/2$ (cf.  \cite[Proof of Proposition 2.12]{HMMM}).
In the presence of ADR property, the result already appears in \cite{DS2,DS1}.

\item  For our purposes, we may ignore those
$k\in \mathbb{Z}$ such that $2^{-k} \gtrsim {\rm diam}(E)$ whenever $E$ is bounded.

\item  We shall denote by  $\mathbb{D}=\mathbb{D}(E)$ the collection of all relevant
$Q^k_j$. That is, $$\mathbb{D} := \bigcup\limits_{k} \mathbb{D}_k,$$
where the union runs only
over those $k$ such that $2^{-k} \lesssim  {\rm diam}(E)$ whenever $E$ is bounded.
%Eventually, in considering domains for which the Corkscrew condition holds, we shall restrict our
%attention to $k$ such that $2^{-k}<(K_0)^{-1}\diam(\partial\Omega)$, in the case that the latter is finite,
%where $K_0$ is some large constant to be chosen depending upon the constants
%in the Corkscrew condition.

\item Given a cube $Q\in \dd$, we set
 %\begin{align}
 %\label{eq2.discretecarl}
\[
\dd_Q:= \left\{Q'\in \dd: Q'\subseteq Q\right\}.
\]
%\end{align}

\item For a dyadic cube $Q\in \mathbb{D}_k$, we
set $\ell(Q) = 2^{-k}$ and we call this quantity the ``length''
of $Q$.  Evidently, $\ell(Q)\approx \diam(Q).$

%\item For a dyadic cube $Q \in \mathbb{D}$, we let $k(Q)$ denote the ``dyadic generation''
%to which $Q$ belongs, i.e., we set  $k = k(Q)$ if
%$Q\in \mathbb{D}_k$; thus, $\ell(Q) =2^{-k(Q)}$.

\item Properties $(iv)$ and $(v)$ imply that for each cube $Q\in\mathbb{D}_k$,
there exists a point $x_Q\in E$, a Euclidean ball $B(x_Q,r_Q)$ and corresponding surface ball
$\Delta(x_Q,r_Q):= B(x_Q,r_Q)\cap E$ such that
\begin{equation}\label{DeltaQ}
c\ell(Q)\leq r_Q\leq \ell(Q)
\qquad\mbox{and}\qquad\Delta(x_Q,2r_Q)\subset Q \subset \Delta(x_Q,Cr_Q)
\end{equation}
for some uniform constants $c$ and  $C$.
We shall denote this ball and surface ball by $B_Q:= B(x_Q,r_Q)$ and $\Delta_Q:= \Delta(x_Q,r_Q)$, respectively,
and we shall refer to the point $x_Q$ as the ``center'' of $Q$.

%\item For any dyadic cube $Q\in\dd$ and for every $\rho>0$ we write
%\begin{equation}\label{defi:dilat-Q}
%(1+\rho)Q:=\{x\in E:\, \dist(x,Q)\leq \rho\diam(Q)\}
%\end{equation}
%to the denote the ``$(1+\rho)$-dilation of $Q$''.

%\item Next we specialize to the case when  $E=\pom$ is ADR,
%with $\Omega$ satisfying the Corkscrew condition. Given $Q\in \mathbb{D}(\partial\Omega)$, we
%shall sometimes refer to a ``corkscrew point relative to $Q$'', which we denote by
%$X_Q$, and which we define to be the corkscrew point $X_\Delta$ relative to the surface ball
%$\Delta:=\Delta_Q$, see \eqref{cube-ball}, \eqref{cube-ball2} and Definition \ref{def1.cork}.  We note that
%\begin{equation}\label{eq1.cork}
%\delta(X_Q) \approx \dist(X_Q,Q) \approx \diam(Q).\end{equation}
\end{list}

It will be useful  to dyadicize the Corkscrew condition and to specify
precise Corkscrew constants.
Let us now specialize to the case that  $E=\pom$ is ADR
with $\Omega$ satisfying the Corkscrew condition.
Given $Q\in \mathbb{D}(\partial\Omega)$, we
shall sometimes refer to a corkscrew point $X_Q$ relative to $Q$, which define to be a corkscrew point $X_\Delta$ relative to the surface ball
$\Delta:=\Delta_Q$ %(see \eqref{cube-ball}, \eqref{cube-ball2} and Definition \ref{def1.cork}).
We note that $\delta(X_Q) \approx \dist(X_Q,Q) \approx \diam(Q)$.
\begin{definition}[{\bf $c_0$-exterior
Corkscrew condition}]
\label{def1.dyadcork}
Fix  a constant $c_0\in (0,1)$ and let
$\Omega\subset \ree$ be a domain with ADR boundary.  We say
that a cube $Q\in \dd(\pom)$ satisfies the
the $c_0$-exterior Corkscrew condition if
there is a point $z_Q\in \Delta_Q$ and a point $X^-_Q\in B(z_Q,r_Q/4) \setminus \overline{\Omega}$ such that
$B(X^-_Q,\,c_0\,\ell(Q))\subset B(z_Q,r_Q/4)\setminus \overline{\Omega}$, where
$\Delta_Q=\Delta(x_Q,r_Q)$ is the surface ball associated to $Q$.
\end{definition}
Following \cite[Section 3]{HM} we next introduce the notion of \textit{\bf Carleson region} and \textit{\bf discretized sawtooth}.  Given a cube $Q\in\dd(\partial\Omega)$, the \textit{\bf discretized Carleson region $\dd_{Q}$} relative to $Q$ is defined by
\[
\dd_{Q}=\{Q'\in\dd(\partial\Omega):\, \, Q'\subset Q\}.
\]
Let $\F$ be family of disjoint cubes $\{Q_{j}\}\subset\dd(\partial\Omega)$. The \textit{\bf global discretized sawtooth region} relative to $\F$ is the collection of cubes $Q\in\dd$  that are not contained in any $Q_{j}\in\F$;
\[
\dd_{\F}:=\dd\setminus \bigcup\limits_{Q_{j}\in\F}\dd_{Q_{j}}.
\]
For a given $Q\in\dd$ the {\bf local discretized sawtooth region} relative to $\F$ is the collection of cubes in $\dd_{Q}$ that are not in contained in any $Q_{j}\in\F$;
\[
\dd_{\F,Q}:=\dd_{Q}\setminus \bigcup\limits_{Q_{j}\in \F} \dd_{Q_{j}}=\dd_{\F}\cap \dd_{Q}.
\]
We also introduce the ``geometric'' Carleson regions and sawtooths. In the sequel, $\Omega \subset \ree$ ($n\geq 2$) will be a 1-sided NTA domain with ADR boundary. Let $\mathcal{W}=\W(\Omega)$ denote a collection
of (closed) dyadic Whitney cubes of $\Omega$  (see \cite[Chapter VI]{St}), so that the cubes in $\mathcal{W}$
form a covering of $\Omega$ with non-overlapping interiors, and  which satisfy
\begin{equation}\label{eqWh1} 4\, {\rm{diam}}\,(I)\leq \dist(4 I,\pom) \leq  \dist(I,\pom) \leq 40 \, {\rm{diam}}\,(I)\end{equation}
and
\begin{equation}\label{eqWh2}\diam(I_1)\approx \diam(I_2), \mbox{ whenever $I_1$ and $I_2$ touch.}\end{equation}
Let $X(I)$ denote the center of $I$, let $\ell(I)$ denote the side length of $I$,
and write $k=k_I$ if $\ell(I) = 2^{-k}$.

Given $0<\lambda<1$ and $I\in\W$ we write $I^*=(1+\lambda)I$ for the ``fattening'' of $I$. By taking $\lambda$ small enough,  we can arrange matters so that, first, $\dist(I^*,J^*) \approx \dist(I,J)$ for every
$I,J\in\W$, and secondly, $I^*$ meets $J^*$ if and only if $\partial I$ meets $\partial J$.
 (Fattening ensures $I^*$ and $J^*$ overlap for
any pair $I,J \in\W$ whose boundaries touch. Thus, the Harnack Chain property holds locally in $I^*\cup J^*$ with constants depending on $\lambda$.)  By picking $\lambda$ sufficiently small, say $0<\lambda<\lambda_0$, we may also suppose that there is $\tau\in(1/2,1)$ such that for distinct $I,J\in\W$, $\tau J\cap I^* = \emptyset$. In what follows we will need to work with dilations $I^{**}=(1+2\,\lambda)I$ and in order to ensure that the same properties hold we further assume that $0<\lambda<\lambda_0/2$.

For every $Q$ we can construct a family $\W_Q^*\subset \W$ and define
\begin{equation}\label{eq2.whitney3}
U_Q := \bigcup_{I\in\,\mathcal{W}^*_Q} I^*\,,
\end{equation}
satisfying the following properties:
$X_Q\in U_Q$ and there are uniform constants $k^*$ and $K_0$ such that
\begin{eqnarray}\label{eq2.whitney2}
& k(Q)-k^*\leq k_I \leq k(Q) + k^*\, \quad & \forall\, I\in \mathcal{W}^*_Q,\\\nonumber
&X(I) \rightarrow_{U_Q} X_Q\,\quad &\forall\, I\in \mathcal{W}^*_Q,\\ \nonumber
&\dist(I,Q)\leq K_0\,2^{-k(Q)}\, \quad &\forall\, I\in \mathcal{W}^*_Q\,.
\end{eqnarray}
Here $X(I) \rightarrow_{U_Q} X_Q$ means that the interior of $U_Q$ contains all the balls in
a Harnack Chain (in $\Omega$) connecting $X(I)$ to $X_Q$, and  moreover, for any point $Z$ contained
in any ball in the Harnack Chain, we have $
\dist(Z,\pom) \approx \dist(Z,\Omega\setminus U_Q)$
with uniform control of the implicit constants.
The constants  $k^*$, $K_0$ and the implicit constants in the condition $X(I)\to_{U_Q} X_Q$ in \eqref{eq2.whitney2}
depend on at most allowable
parameters and on $\lambda$. The reader is referred to \cite{HM} for full details.

For a given $Q\in\dd$, the {\bf Carleson box} relative to $Q$ is defined by
\[
T_{Q}:=\mbox{int}\left(\bigcup\limits_{Q'\in\dd_{Q}} U_{Q'}\right).
\]
For a given family $\F$ of disjoint cubes $\{Q_{j}\}\subset\dd$, the {\bf global sawtooth region} relative to $\F$ is
\[
\Omega_{\F}:=\mbox{int}\left(\bigcup\limits_{Q'\in\dd_{\F}} U_{Q'}\right).
\]
Finally, for a given $Q\in\dd$ we define the {\bf local sawtooth region} relative to $\F$ by
\[
\Omega_{\F,Q}:=\mbox{int}\left(\bigcup\limits_{Q'\in\dd_{\F,Q}}U_{Q'}\right).
\]
For later use we recall \cite[Proposition 6.1]{HM}:
 \begin{align}
\label{prop61}
Q\setminus \bigg( \bigcup_{Q_{j}\in \F} Q_{j}\bigg) \subset \partial\Omega \cap\partial\Omega_{\F,Q}
\subset
\overline{Q}\setminus \bigg( \bigcup_{Q_{j}\in \F} \interior(Q_{j})\bigg).
 \end{align}

Given a pairwise disjoint family $\F\in\dd$  and a constant $\rho>0$, we derive another family $\F({\rho})\subset\dd$  from $\F$ as follows. Augment $\F$ by adding cubes $Q\in\dd$ whose side length $\ell(Q)\leq \rho$ and let $\F(\rho)$ denote the corresponding collection of maximal cubes. Note that the corresponding discrete sawtooth region $\dd_{\F(\rho)}$ is the union of all cubes $Q\in\dd_{\F}$ such that $\ell(Q)>\rho$.
For a given constant $\rho$ and a cube $Q\in \dd$, let $\dd_{\F(\rho),Q}$ denote the local discrete sawtooth region and let $\Omega_{\F(\rho),Q}$ denote the geometric sawtooth region relative to disjoint family $\F(\rho)$.

\section{Proof of Theorem \ref{main}}\label{sect:proof-main}

\subsection{Proof of \eqref{maina} implies \eqref{mainb}}
\label{rectandapprox}

Our goal in this section is to prove that \eqref{maina} implies \eqref{mainb} in Theorem \ref{main}. To this purpose, we first recall a useful notion from geometric measure theory. For any affine $n$-plane $\mc{P}$ and $\eta>0$, set $\mc{P}(\eta):=\{X:\, \dist(X,\mc{P})\leq \eta\}$.
\begin{definition}[{\bf Linear approximation}]
\label{approximable}
A set $E$ in $\mathbb{R}^{n+1}$ is called $n$-linearly approximable if for $\rh^{n}$-a.e.~$a\in E$ the following holds: if $\eta$ is a positive number, there are positive numbers $r_{0}, \lambda$ and an affine $n$-plane $\mc{P}$ such that $a\in \mc{P}$ and for any $0<r<r_{0}$,
\begin{list}{$(\theenumi)$}{\usecounter{enumi}\leftmargin=.8cm
\labelwidth=.8cm\itemsep=0.2cm\topsep=.1cm
\renewcommand{\theenumi}{\roman{enumi}}}

\item $\rh^{n}(E\cap B(X, \eta r))\geq \lambda r^{n}$, for every $X\in \mc{P}\cap B(a,r)$;

\item  $\rh^{n}(E\cap B(a,r)\setminus \mc{P}(\eta r))<\eta r^{n}$.
\end{list}
\end{definition}
See Figure \ref{linearapproxfigure} for visualization of Definition \ref{approximable}.
\def\firstcircle{(0,0) circle (4cm)}
\def\secondcircle{(2,0) circle (1cm)}
\def\thirdcircle{(0,0) circle (4cm)}
\def\fourthcircle{(0,0) circle (4cm)}

\begin{figure}[!ht]
\centering
\begin{tikzpicture}[scale=.5]
\draw[name path=C1, fill=gray] \firstcircle;
\fill[white] (-6,-1) rectangle (6,1);
%\node[below] at (0,-4) {$B(0,r)$};
\node[above] at (2,1) {$B(x,\eta r)$};
\draw[name path=L1] (-5,1)--(5.5,1);
\draw[name path=L2] (-5,-1)--(5.5,-1);
\path [name intersections={of=C1 and L1, total=\n}];
\draw[red] \secondcircle;
\draw[blue] (-6,0)--(5,0);
\draw [magenta] plot [smooth] coordinates {(-3,1.8) (-2.5,2) (-2,2)};
\draw [magenta, shift={(2 cm, 2 cm)},rotate=20] plot [smooth] coordinates {(-3,1.8) (-2.5,2) (-2,2)};
\draw [magenta, shift={(3.5 cm, 1 cm)},rotate=10] plot [smooth] coordinates {(-3,1.8) (-2.5,2) (-2,2)};
\draw [magenta, shift={(5 cm, 1 cm)}] plot [smooth] coordinates {(-3,1.8) (-2.5,2) (-2,2)};
\draw [magenta, shift={(-2 cm, 0 cm)}] plot [smooth] coordinates {(-3,1.8) (-2.5,2) (-2,2)};
\begin{scope}[xshift=-1cm, yshift=-5cm]
\draw [magenta] plot [smooth] coordinates {(-3,1.8) (-2.5,2) (-2,2)};
\draw [magenta, shift={(2 cm, 2 cm)},rotate=20] plot [smooth] coordinates {(-3,1.8) (-2.5,2) (-2,2)};
\draw [magenta, shift={(3.5 cm, 1 cm)},rotate=10] plot [smooth] coordinates {(-3,1.8) (-2.5,2) (-2,2)};
\draw [magenta, shift={(5 cm, 1 cm)}] plot [smooth] coordinates {(-3,1.8) (-2.5,2) (-2,2)};
\draw [magenta, shift={(-2 cm, 0 cm)}] plot [smooth] coordinates {(-3,1.8) (-2.5,2) (-2,2)};
\end{scope}
\draw[->] (0,-2.5)--((4,-4) node[below, gray] at (4,-4) {The surface measure of the portions
of $E$};
\node[below, gray] at (4,-5) {in the gray area is smaller than $\eta r^{n}$};
\node[left,magenta] at (-5,2) {$E$};
\node[left,blue] at (-6,0) {$\mc{P}$};
\filldraw (2,0) circle (1pt);
\node[below] at (2,0){$x$};
\filldraw (0,0) circle (1pt);
\node[below] at (0,0){$a$};
\draw[->] (5.5,.3) -- (5.5,.8);
\draw[->] (5.5,-.3) -- (5.5,-.8);
\node at (5.5,0) {\qquad\ $2 \eta r $};
\node[magenta] at (0,0){\large Big chunk of $E$};
\end{tikzpicture}
\caption{Linear approximation of set $E$.}
\label{linearapproxfigure}
\end{figure}
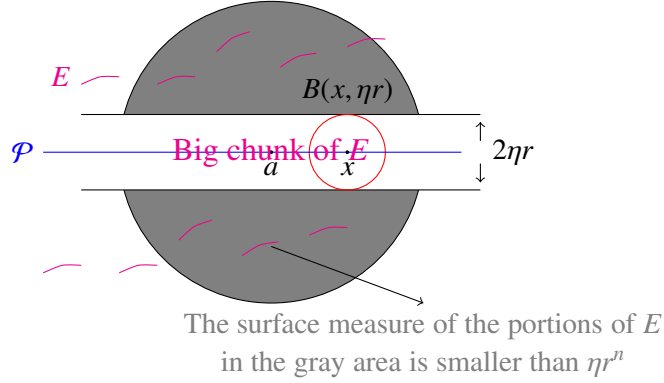
%%
%%
%\begin{center}
%\includegraphics[page=1,scale=0.3]{pictures.pdf}
%\end{center}
%%
%%
\begin{lemma}[{\cite[Theorem 15.11]{Ma95}}]
\label{Ma15.11}
If $E$ is an $\rh^{n}$ measurable $n$-rectifiable subset of $\mathbb{R}^{n+1}$ with $\rh^{n}(E)<\infty$, then $E$ is $n$-linearly approximable.
\end{lemma}

We want to show that  $n$-rectifiability implies existence of two-sided corkscrews in the presence of Ahlfors-David regularity.
Fix any $n$-rectifiable, $n$-dimensional ADR set $E\subset\mathbb{R}^{n+1}$. Let us observe that the ADR condition implies that $E$ is $\rh^n$-locally finite and then $E$ is $n$-linearly approximable by Lemma \ref{Ma15.11}
(note that the $n$-linear approximability is a local property and hence   Lemma \ref{Ma15.11} immediately extends
to any $E$ having locally finite measure).  Thus, we may fix $a\in E$ for which $(i)$ and $(ii)$ in Definition \ref{approximable} holds for some $0<\eta\ll 1/4$ to be chosen and some constants $r_0$ and $\lambda$ and $n$-plane $\mc{P}$ depending on $E$, $a$ and $\eta$. After a harmless rotation and translation, we assume $a=0$ and $\mc{P}=\{x\in \mathbb{R}^{n+1}:x_{n+1}=0\}$ with $0\in \mc{P}$. Let $N\ge 1$ be a large constant to be chosen and fix $0<r<\eta\, r_0$. Set $B=B(0,r)$, $\zeta=\eta^{\frac{n}{n+1}}$, and let
$$
B^{\pm}
=\{X=(X_1,\dots, X_{n+1})\in B: \pm X_{n+1}>4\,\zeta\,r\}
$$
denote the upper and lower parts of $B\setminus \mc{P}(4\,\zeta\, r)$. We also set
\[
\Sigma =\left\{X\in\mathbb{R}^{n+1}\setminus E: \, \dist(X,E)  <\frac{4\,\zeta\,r}{N} \right\}.
\]
See Figure \ref{setBpBmSfigure}.

\def\firstcircle{(0,0) circle (4cm)}
\def\secondcircle{(2,0) circle (1cm)}
\def\thirdcircle{(0,0) circle (4cm)}
\def\fourthcircle{(0,0) circle (4cm)}

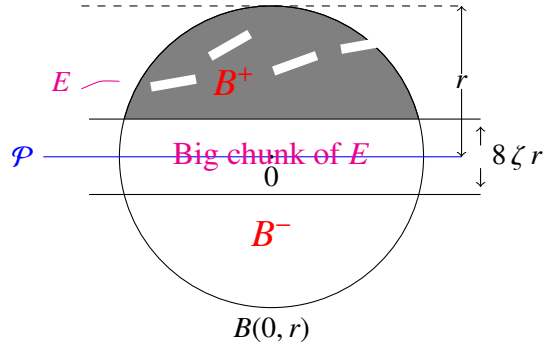
\begin{figure}[!ht]
\centering
\begin{tikzpicture}[scale=.5]
\draw[name path=C1, fill=gray] \firstcircle;
\fill[white] (-6,-5) rectangle (6,1);
\draw \firstcircle;
\node[below] at (0,-4) {$B(0,r)$};
\draw[name path=L1] (-4.8,1)--(5.5,1);
\draw[name path=L2] (-4.8,-1)--(5.5,-1);
\path [name intersections={of=C1 and L1, total=\n}];
\draw[blue] (-6,0)--(5,0);
\draw [magenta] plot [smooth] coordinates {(-3,1.8) (-2.5,2) (-2,2)};
\fill[white, rotate=10] (-2.8, 2.25) rectangle (-1.6,2.5);
\draw [magenta, shift={(2 cm, 2 cm)},rotate=20] plot [smooth] coordinates {(-3,1.8) (-2.5,2) (-2,2)};
\fill[white, shift={(2 cm, 2 cm)}, rotate=30] (-2.8, 2.25) rectangle (-1.6,2.5);
\draw [magenta, shift={(3.5 cm, 1 cm)},rotate=10] plot [smooth] coordinates {(-3,1.8) (-2.5,2) (-2,2)};
\fill[white, shift={(3.5 cm, 1 cm)}, rotate=20] (-2.8, 2.25) rectangle (-1.6,2.5);
\draw [magenta, shift={(5 cm, 1 cm)}] plot [smooth] coordinates {(-3,1.8) (-2.5,2) (-2,2)};
\fill[white, shift={(5 cm, 1 cm)}, rotate=10] (-2.8, 2.25) rectangle (-1.6,2.5);
\draw [magenta, shift={(-2 cm, 0 cm)}] plot [smooth] coordinates {(-3,1.8) (-2.5,2) (-2,2)};
\node[left,magenta] at (-5,2) {$E$};
\node[left,blue] at (-6,0) {$\mc{P}$};
\draw[->] (5.5,.3) -- (5.5,.8);
\draw[->] (5.5,-.3) -- (5.5,-.8);
\node at (5.5,0) {$\qquad\ \ 8\,\zeta\, r $};
\node[red] at (-1,2) {{\Large $B^{+}$}};
\draw[dashed] (-5,4)--(5,4);
\draw[->] (5,2.2)--(5,4);
\draw[->] (5,1.8)--(5,0);
\node at (5,2){$r$};
\filldraw (0,0) circle (1pt);
\node[below] at (0,0){$0$};
\node[red] at (0,-2) {{\Large $B^{-}$}};
\node[magenta] at (0,0){\large Big chunk of $E$};
\end{tikzpicture}
\caption{Sets $B^{+}$, $B^{-}$, $\Sigma$.}
\label{setBpBmSfigure}
\end{figure}

\begin{lemma}\label{lemma:rect-two-points}
We can take $N\ge 1$ large enough and $0<\eta\ll 1/4$ small enough depending only on the ADR constants and dimension so that if $B:=B(0,r)$ for some $0<r<\eta\,r_0$, then there exist $X^{\pm}\in B^{\pm}\setminus \Sigma $. Hence $B(X^{\pm}, 2\,\zeta\,r/N)\subset B(0,2\,r)\setminus E$, where $\zeta=\eta^{\frac{n}{n+1}}$.
%Consequently, for every  $x'\in E\cap B(0,r/2)$  one has $B(X^{\pm}, 2\,\eta\,r/N)\subset B(x',2\,r)\setminus E$. Hence $X^{\pm}\in\ree\setminus E$ are corkscrew points relative to $B(x',2r)\cap E$ with constant $\eta\,r/N$.
\end{lemma}

\begin{proof}
We work with $B^+$ (the proof for $B^-$ is identical). Our aim is to show that
\begin{align}
\label{corkscrowexists}
\left|B^{+}\setminus \Sigma \right|\gtrsim r^{n+1}.
\end{align}
It then easily follows that $B^{+}\setminus\Sigma \neq \emptyset$ and we can pick any $X^+\in B^{+}\setminus\Sigma$.

To show \eqref{corkscrowexists}, we let $\W=\mc{W}(\ree\setminus E)$ be the Whitney decomposition of $\ree\setminus E$ and set $\W^{k}=\{I\in \W :\, \,\ell(I)=2^{-k}\}$.
Let $I\in \W$ be such that $I\cap B_+\cap \Sigma  \neq\emptyset$ and pick $Y=Y_I=(Y_{1},\ldots,Y_{n+1})\in I\cap B_+\cap \Sigma $.
Then $\ell(I)\approx \dist(I,E)\le \dist(Y,E)<4\,\zeta\,r/N$ and hence $I\in \W^k$ for some $k$ verifying $2^{-k}\lesssim \zeta\,r/N$.

From the definition of the dyadic grid we can associate to each $I\in \W $ a nearest dyadic cube $Q_{I}\in \mathbb{D}(E)$ such that
\[
\ell(Q_{I})=\ell(I)\qquad\mbox{and}\qquad \dist(I, Q_{I})=\dist(I,E)\approx\ell(I).
\]
(Just pick one if there are several choices available.) For every $y\in Q_{I}$, we have $|y-Y|\approx\dist(I,E)\approx\ell(I)\lesssim \zeta\, r/N$. In particular, we have $|y_{n+1}-Y_{n+1}|\lesssim \zeta r/N$.
Since $Y\in B^+$ (i.e., $Y_{n+1}>4\,\zeta\,r$), taking $N$ large enough depending on the ADR constants and dimension, we conclude that $y_{n+1}>2\,\zeta\,r$. Thus, choosing  $\eta$ (and hence $\zeta$) small enough depending on the ADR constants and dimension, we obtain
$$
Q_{I}
\subset
E\cap B(0,2\,r)\setminus\P(2\,\zeta\,r)
\subset
E\cap B(0,2\,\zeta\,\eta^{-1}\,r)\setminus\P(\eta\,(2\,\zeta\,\eta^{-1}\,r)).
$$
Note that with $k\in \Z$ fixed, the family $\{Q_I\}_{I\in \W^k}$ has bounded overlap (with overlap independent of $k$). Therefore,
\begin{multline*}
\sum\limits_{I\in \W^k} \ell(I)^n
\approx
\sum\limits_{I\in \W^k} \sigma(Q_{I})
\lesssim
\sigma\left(\bigcup_{I\in \W^k}Q_{I}\right)
\\
\leq
\sigma\big(E\cap B(0,2\,\zeta\,\eta^{-1}\,r)\setminus\P(\eta\,(2\,\zeta\,\eta^{-1}\,r))\big)
<
\eta\,\big(2\,\zeta\,\eta^{-1}\,r\Big)^n
=
2^n\,\eta^{1-n}\,\zeta^{n}\,r^n,
\end{multline*}
where in the last estimate we have used $(ii)$ in Definition \ref{approximable} along the fact that $0<r<\eta\,r_0$ and that $\eta$ is small enough (depending on $n$). We conclude that
\begin{multline}
\label{IsumB1}
|B^{+}\cap\Sigma |
=
\sum_{I\in \W}|I\cap B_+\cap \Sigma |
\le
\sum\limits_{k:\, 2^{-k}\lesssim \frac{\zeta\,r}{N}}\, \,  \sum\limits_{I\in\W^k} \ell(I)^{n+1}
\\
\le
\sum\limits_{k:\, 2^{-k}\lesssim \frac{\zeta\,r}{N}}\, 2^{-k}\,  \sum\limits_{I\in\W^k} \ell(I)^{n}
\lesssim
\eta^{1-n}\,\zeta^{n+1}\,r^{n+1}
=\eta\,r^{n+1}
.
\end{multline}
This and \cite[Lemma 5.3]{HM} easily imply
\begin{multline*}
r^{n+1}
\approx
\big|B\cap\{X\in\ree:0<X_{n+1}\le 4\,\zeta\,r\}\big|
+
|B^+\cap\Sigma |
+
|B^+\setminus\Sigma |
\\
\lesssim
(\zeta+\eta)\,r^{n+1}+|B^+\setminus\Sigma |.
\end{multline*}
Taking now $\eta>0$ small enough  depending only on ADR constants and dimension, we can hide the first time in the last term and conclude as desired \eqref{corkscrowexists}.
\end{proof}

We are now ready to establish the main result of this section:

\begin{proposition}
\label{corkomega}
Let $\Omega$ be a 1-sided NTA domain with ADR boundary and assume that $\pom$ is $n$-rectifiable. There exists $0<c<1$ depending on the 1-sided NTA and ADR constants  such that for $\sigma$-a.e.~$x\in\pom$ there is a scale $r_{x}>0$ such that for all $0<r<r_x$ there exist $X_{\Delta(x,r)}^{\rm int}, X_{\Delta(x,r)}^{\rm ext} \in B(x,r)$ that are respectively interior and exterior corkscrew points relative to $\Delta(x,r)$ with implicit constant $c$.
\end{proposition}

\begin{proof}
We can use Lemma \ref{Ma15.11} to find a subset of $E=\pom$ with full $\sigma$-measure on which  $(i)$ and $(ii)$ of Definition \ref{approximable} hold. We can make the previous reductions and find $X^{\pm}$ as in Lemma \ref{lemma:rect-two-points}  associated with $B:=B(0,r)$ whenever $0<r<\eta\,r_0$, where $N\ge 1$ is a fixed small number and $\eta$ is small enough but at our disposal.

We claim that if $\eta$ is small enough depending on the 1-sided NTA and ADR constants, then at least one of $X^{\pm}$ belongs to $\Omega_{\rm ext}$. Suppose otherwise that $X^{\pm}\in \Omega$ (by construction $X_\pm\notin\pom$). Since $\Omega$ is a 1-sided NTA domain, $\Omega$ is also a uniform domain. In fact, the two notions are equivalent; see \cite{AHMNT} for the definition of a uniform domain and for a proof of the direction that is relevant here. Thus, there exist $0<c_1<1$ and $C_1>1$ depending only the 1-sided NTA constants and a path $\gamma$ connecting $X^+$ and $X^{-}$ in $\Omega$ so that
$$
\ell(\gamma)\le C_1\,|X^- - X^+|
\quad\mbox{and}\quad
\dist(Z,\pom)\ge c_1\,\dist(Z, \{X^-, X^+\}) \quad \forall\,Z\in\gamma.
$$ In the previous expression, $\ell(\gamma)$ denotes the length of $\gamma$. For every $Y\in\gamma$, we have
$$
|Y|
\le
|Y-X^+|+ |X^+|
\le
\diam(\gamma)+r
\le
\ell(\gamma)+r
\le C_1\,|X^- - X^+|+r
\le (2\,C_1+1)\,r,
$$
since $X^+\in\gamma$. Hence $\gamma\subset B(0,(2\,C_1+1)\,r)$. On the other hand $X^{\pm}\in B^{\pm}$. Hence $X^+$ lies above $\P$ and $X^{-}$ lies below $\P$. In particular, we can thus find $Z\in\P\cap\gamma\subset \P\cap B(0,(2\,C_1+1)\,r)$. If we assume that $(2\,C_1+1)\,r<r_0$, then we can apply $(i)$ in Definition \ref{approximable} to find $z\in\pom$ such that $\dist(Z,\pom)\le |Z-z|<\eta\,(2\,C_1+1)\,r$. Also note that since $Z\in \P$ and $X^{\pm}\in B^{\pm}$,
$$
|Z-X^{\pm}|
\ge
|X_{n+1}^{\pm}|
>4\,\zeta\,r
=
4\,\eta^{\frac{n}{n+1}}\,r.
$$
Hence
$$
\eta\,(2\,C_1+1)\,r
>
\dist(Z,\pom)
\ge
c_1\,\dist(Z, \{X^-, X^+\})
\ge
c_1\,4\,\eta^{\frac{n}{n+1}}\,r.
$$
We can clearly take $\eta$ arbitrarily small depending only on $C_1$, $c_1$ and $n$ so that the previous estimate does not hold and this brings us to a contradiction.

Let us summarize the argument so far. We can pick  $\eta_0$ small enough (depending on the 1-sided NTA and ADR constants) so that if $$0<r<r_x:=r_0\,\min\{\eta_0, (2\,C_1+1)^{-1}\},$$ then $X^{+}$ or $X^{-}$ is in $\Omega_{\rm ext}$. Let $X^{\rm ext}$ denote one of the points in $\Omega_{\rm ext}$. By Lemma \ref{lemma:rect-two-points},
$B(X^{\rm ext}, 2\,\zeta_0\,r/N))\subset B(0,2\,r)\setminus\pom$ and hence $X^{\rm ext}$ is an exterior corkscrew point relative to $\Delta(0,2\,r)$ with implicit constant $\zeta_0/N$. On the other hand since $\Omega$ is a 1-sided NTA domain it satisfies the (interior) corkscrew condition and hence we can find $X^{\rm int}$ an interior corkscrew point relative to $\Delta(0,2\,r)$ with implicit constant $c_0$. This readily leads to the desired conclusion with $r_x$ as above and $c=\min\{\zeta_0/N,c_0\}$. This completes the proof. 
\end{proof}

\subsection{Proof of \eqref{mainb} implies \eqref{maind}}
\label{bimpliesd}
In this section, we prove \eqref{mainb} implies \eqref{maind}. Suppose there exist a Borel measurable set $\mathbb{F}_0\subset\pom$ with $\sigma(\mathbb{F}_0)=0$ and  constant $0<2\,c_0<1$ such that:
\begin{quotation}
For each $x\in\pom\setminus \mathbb{F}_0$, there is a scale $0<r_x<\diam(\pom)$  such that for every $0<r<r_x$ there exist interior and exterior corkscrew points relative to $\Delta(x,r)$ with implicit constant $2\,c_0$.\end{quotation}
By taking $r_x$ smaller if needed, we may assume that $r_x=2^{-k_x}$ for some $k_x\in\Z$. Given $k\in\Z$ we consider the closed set (and therefore measurable set)
$$
E_{k}:=\overline{\{x\in \pom\setminus\mathbb{F}_0:\, \, r_{x}= 2^{-k}\}}.
$$
Then $ \partial\Omega= \mathbb{F}_0\cup \bigcup_{k\in\Z} E_k$.
In turn, for each $k$ we can write  $E_k =\bigcup_{Q\in\dd_{k}} E_k\cap Q$. To establish \eqref{maind}, it suffices to show that for every $k\in\Z$ and $Q\in\dd_{k}$, there
exists a bounded-chord arc domain $\Omega_\star\subset\Omega$ such that $E_k\cap Q\subset \pom\cap\partial\Omega_\star$.

Fix $k\in\Z$ and $Q_0\in\dd_k$ for which $E_k\cap Q_0\neq\emptyset$. Suppose $x\in E_k$ and $0<r<2^{-k}$. Then there exists $y\in \pom\setminus\mathbb{F}_0$ such that $r_y=2^{-k}$ and $|x-y|<r/2$. Let $X^{\pm}$ be interior/exterior corkscrew points relative to $\Delta(y,r/2)$ with implicit constant $2c_0$ and note that
$$
B(X^{\pm}, c_0\,r)
=
B(X^{\pm}, 2\,c_0\,(r/2))
\subset B(y,r/2)\cap\Omega^{\pm}\subset B(x,r)\cap\Omega^{\pm},
$$
where $\Omega^+$ and $\Omega^-$ denote $\Omega$ and $\Omega_{\rm ext}$, respectively. We conclude that for every $x\in E_k$ and for every $0<r<2^{-k}$ there are interior/exterior corkscrew points $X^\pm$  relative to $\Delta(x,r)$ with implicit constant $c_0$. This is the key property for the rest of the argument in this section. To continue, set $F_k=E_k\cap Q_0$ and dyadically subdivide $Q_{0}$, stopping whenever $Q\cap F_k=\emptyset$. If we never stop, set $\F=\emptyset$. Otherwise, $\F=\{Q_{j}\}_{j\geq 1}\subset \dd_{Q_0}\setminus Q_0$ is the pairwise collection of stopping time dyadic cubes and it follows that $Q_j\cap F_k=\emptyset$ and $Q\cap F_k\neq\emptyset$ whenever $Q_j\subsetneq Q\subset Q_0$. We remark that we did not stop at $Q_0$, because $F_k\neq\emptyset$. Also, $F_k=Q_0\setminus \bigcup_{j\geq 1} Q_j$, because $E_k$ is a closed set.

Set $\Omega_\star=\Omega_{\F,Q_0}$ (where it is understood that $\Omega_\star=T_{Q_0}$ if $\F=\emptyset$). Then $\Omega_\star$ is a bounded 1-sided NTA domain with ADR boundary by \cite[Lemma 3.61]{HM}, where all implicit constants for $\Omega_\star$ depend only on the corresponding constants for $\Omega$.
By \eqref{prop61}, we only need to check that $\Omega_\star$ satisfies the exterior corkscrew condition, which will follow from the definition of the set $F_k$.

To complete the proof, let $M>1$ denote a large constant to be chosen below. Fix any boundary point $x\in\pom_\star$ and any scale $0<r<2^{-k}=\ell(Q_0)\approx\diam(\pom_\star)$, and set $\Delta_\star=B(x,r)\cap\pom_\star$. We consider two cases:

\noindent{\bf Case 1:} Suppose that $0\le \delta(x)\le  r/M$, where $\delta(x)=\dist(x,\partial\Omega)$. We first claim that there exists $Q\in\dd_{Q_0}$ with $\ell(Q)\approx r/M$ such that $|x-x_Q|\lesssim r/M$. To see this, note that on the one hand, if $x\in\partial\Omega_{\star}\cap \partial\Omega$, then $x\in \overline{Q_0}$ (see \eqref{prop61}) and we can find $Q\in \dd_{Q_0}$ with $\ell(Q)\approx r/M$ and $x\in\overline{Q}$. On the other hand, if $x\in\partial\Omega_\star\cap\Omega$, then by the definition of the sawtooth region, $x\in\partial I^*$ for some $I\in\W_{Q'}^*$ with $Q'\in\dd_{\F,Q_0}$ and $|x-x_{Q'}|\approx\dist(I,Q')\approx\ell(Q') \approx \ell(I)\approx\delta(x)\le r/M$. Let us now take $Q\in\dd_{Q_0}$, an ancestor of $Q'$, such that $\ell(Q)\approx r/M$ and hence $|x-x_Q|\le |x-x_{Q'}|+|x_{Q'}-x_Q|\lesssim r/M$. This verifies the claim.

Take $Q$ as in the claim and consider two cases. Suppose first that there exists $y\in Q\cap F_k\neq\emptyset$. Then,  as shown above, there is  $X^{-}$ such that $B(X^{-},c_0\,r/2)\subset B(y,r/2)\cap\Omega_{\rm ext}$. Therefore, for $M$ large enough, we have
$$
B(X^{-},c_0\,r/2)
\subset
B(y,r/2)\cap\Omega_{\rm ext}
\subset
B(x,r)\cap(\Omega_\star)_{\rm ext}.
$$
Suppose otherwise that $Q\cap F_k=\emptyset$. Then $Q\subset \bigcup_{j\geq 1} Q_j$, say $Q\cap Q_{i}\neq\emptyset $ for some $Q_{i}\in\F$.
Recall that $\widehat{Q}_{i}$, the father of $Q_{i}$, meets $F_k$ and therefore $Q\subset Q_{i}$. Thus, by \cite[Lemma 5.9]{HM}, there is a ball $B'\subset\ree\setminus \Omega_\star$ whose center is $x_{Q}$ and whose radius is of the order of $\ell(Q)\approx r/M$. For $M$ large enough, this gives the desired exterior corkscrew condition relative to $\Delta_\star$. This completes Case 1.

\noindent{\bf Case 2:} Suppose that $\delta(x)> r/M$. Then $x\in\Omega$, and by definition of the sawtooth region, $x\in\partial I^*\cap J$ for some Whitney cubes $I\in\W_Q^*$, $Q\in \dd_{\F,Q_0}$,  and $J\in\W$ with $\tau\,J\subset\Omega\setminus\Omega_\star$ for some $\tau\in (1/2,1)$. Note that $\ell(I)\approx\ell(J)\approx\delta(x)>r/M$. Hence we can easily find an exterior corkscrew in the segment joining $x$ with the center of $J$ with corkscrew constant that depends only on $M$ and the implicit constants in the previous estimates. This completes Case 2.

This finishes the proof of \eqref{mainb} implies \eqref{maind}.

\subsection{Proof of \eqref{maind} implies (\ref{mainc})}
\label{dimpliesc}

The argument is quite simple. Let $F\subset\pom$ be a Borel set and suppose that $\omega(F)=0$. Then $\omega(F\cap F_N)=0$ for every $N$ and it suffices to show that $\sigma(F\cap F_N)=0$ for each $N$. Fix $N$ and write $\omega_N$ for the harmonic measure of $\Omega_N$ with pole at $X_N$, any point of $\Omega_N$. By Harnack's inequality
and the maximum principle (see the justification below), $0\le \omega_N(F\cap F_N)\le \omega^{X_N}(F\cap F_N)=0$. But $\omega_N\in A_\infty(\pom_N)$, since $\Omega_N$ is a chord-arc domain (see \cite{DJ, Sem}). Hence $\omega_N$ and $\rh^n\big|_{\pom_N}$ are mutually absolutely continuous. Therefore, $\sigma(F\cap F_N)=0$, because $\omega_N(F\cap F_N)=0$.

Let us justify the use of maximum principle. One can use Perron's method (see \cite[Chapter 2]{GT} for more details) to easily see that every superfunction relative to $\chi_{F\cap F_N}$ for $\Omega$ is also a
superfunction relative to $\chi_{F\cap F_N}$ for $\Omega_N$,  since $\Omega_N\subset\Omega$ and $F\cap F_N\subset \pom\cap\pom_N$. Hence the desired inequality follows after taking the infimum over such superfunctions. This works for the Laplacian and does not require Wiener regularity. However, since below we are also interested in the case of variable coefficients, we now present a more robust, alternative argument, borrowed from \cite{HMT-general}.  Fix a compact set $\widetilde{F}\subset F\cap F_N$  and a small error $\epsilon>0$.  Since $\omega^{X_N}$ is outer regular, there exists a bounded, relatively open set $U\subset\pom$ such that $\widetilde{F}\subset U$ and
$$\omega^{X_N}(U)\le \omega^{X_N}(\widetilde{F})+\epsilon.
$$
By Urysohn's lemma there exists $\varphi\in C_c(\pom)$ such that $0\le \varphi\le 1$, $\varphi\equiv 1$ on $\widetilde{F}$ and $\varphi\equiv 0$ on $\pom\setminus U$. Let $u$ denote the Poisson extension of $\varphi$, i.e.,
$$
u(X)
:=
\int_{\pom}\varphi(y)\, d\omega^{X}(y),
\qquad \forall\,X\in\Omega.
$$
Because $\partial\Omega$ is ADR, every $x\in\partial\Omega$ is regular in the sense of Wiener. Indeed, by the dual characterization of capacity using Wolff's potential, see \cite[Theorems  2.2.7 and 4.5.2] {AH} and \cite[Theorem 2.38 and Example 2.12]{HKM},  one can see that ADR yields that $\Omega$ satisfies the ``capacity density condition'' (see \cite{Aikawa} for the precise defintion) which, in turn, is a stronger quantitative version of Wiener regularity (details can be found in the forthcoming papers \cite{HLMN, HMT-general}). Hence $u\in C(\overline{\Omega})$, where $u|_{\partial\Omega}=\varphi$, and thus, $u\in C(\overline{\Omega_{N}})$, as well. It follows that
\begin{equation}\label{e:4-1}
\omega_{N}(\widetilde{F})
=
\int_{\partial \Omega_{N}} 1_{\widetilde{F}}(y)\, d\omega_{N}(y)
\le
\int_{\partial \Omega_{N}} u(y)\, d\omega_{N}(y)
=
u(X_N),\end{equation} where the last equality holds
by the strong maximum principle and the fact that $\Omega_N$ is bounded. On the other hand,
\begin{equation}\label{e:4-2}
%\omega^{X_N}(\widetilde{F})
%=
%\int_{\pom} 1_{\widetilde{F}}(y)\, d\omega^{X_N}(y)
%\le
u({X_N})
\le \int_{\partial\Omega} \chi_U(y)\,d\omega^{X_N}(y)=
\omega^{X_N}(U)
\le
\omega^{X_N}(\widetilde{F})+\epsilon.\end{equation} Combining \eqref{e:4-1} and \eqref{e:4-2} and letting $\epsilon\rightarrow 0$, we conclude that $\omega_N(\widetilde F) \leq \omega^{X_N}(\widetilde F)$ for every compact set $\widetilde F\subset F\cap F_N$ Therefore, since $\omega_N$ and $\omega^{X_N}$ are inner regular,
$\omega_N(F\cap F_N) \leq \hm^{X_N}(F\cap F_N)$, as claimed above.

\subsection{Proof of \eqref{mainc} implies \eqref{maind}}
\label{cimpliesd}
In this section we prove \eqref{mainc} implies \eqref{maind}. Assume that $\sigma\ll \omega$. Fix $Q_{0}\in\dd_{k_0}$ where $k_0\in\Z$ is taken so that  $2^{-k_0}\ll\diam(\pom)$. From the construction of $T_{Q_0}$ one can easily see that $T_{Q_0}\subset \kappa_0\,B_Q:=B_{Q_0}^*$,
where $\kappa_0$ is a constant depending on the ADR and 1-sided NTA constants and the parameters in \eqref{eq2.whitney2} (see \cite{HM}).
Let $X_{0}$ be an interior corkscrew point for $\kappa\,\Delta_{Q_0}^*$ where $\kappa$ is a large,  but fixed constant, for which $X_0\notin 4\, B_{Q_0}^*$. Note that implicitly, we need $\ell(Q_0)\ll \diam(\pom)$. Since $\partial\Omega$ is ADR, Bourgain's alternative \cite{B87} implies that there exist $0<c<1$ and $C>1$ depending only on $n$ and ADR such that for every $x\in \pom$ and $0<r<\diam(\pom)$ one has that $\omega^Y(\Delta(x,r))\ge C^{-1}$ for every $Y\in \Omega\cap B(x,c\,r)$. This and Harnack's inequality gives $\omega^{X_{0}}(Q_0)\ge C_0^{-1}$, where $C_0\ge 1$ depends on ADR constants and $\kappa$. Thus,  $\omega:=C_0\,\sigma(Q_{0})\,\omega^{X_{0}}$ satisfies
\begin{align}
\label{normalizedomega*}
1\leq\frac{\omega(Q_{0})}{\sigma(Q_{0})}\leq C_0.
\end{align}
% and let $\mc{M}_{Q}^{d}(\cdot)$ be the dyadic maximal operator associated with $\dd_Q$. We first have from the well-known property of the dyadic maximal operator and from  our normalization that %the assumption $\sigma\ll \omega$ that
%\begin{align*}
%%\label{sigma1N}
%%\begin{split}
%&\sigma\left(x\in Q_{0}:\, \, \mc{M}_{Q_{0}}^{d}(\omega)>N\right)\leq \frac{1}{N}\omega(Q_{0})\leq\frac{1}{N}\sigma(Q_{0}),
%%\\
%%&\omega\left(x\in Q_{0}:\, \, \mc{M}_{\omega,Q_{0}}^{d}(\sigma)>K(N)\right)\leq \frac{1}{K(N)}\sigma(Q_{0}).
%%\end{split}
%\end{align*}
%when $N$ is sufficiently large.
Let $N\ge C_0$ and let $\F_N=\{Q_j\}\subset \dd_{Q_0}\setminus\{Q_0\}$ be the collection of descendants of $Q_0$ that are maximal with respect to the property that either
\begin{align}
\label{Q0maxsubd}
\frac{\omega(Q_{j})}{\sigma(Q_{j})}<\frac{1}{N}\qquad \mbox{or}\qquad\frac{\omega(Q_{j})}{\sigma(Q_{j})}>N.
\end{align}
By maximality, it follows that
\begin{align}
\label{omegasigmabdd}
\frac{1}{N}\le \frac{\omega(Q)}{\sigma(Q)}\leq N
\qquad
\forall\, Q\in\dd_{\F_{N},Q_{0}}.
\end{align}
On the other hand, we can write
\begin{equation}
\label{eqn:cover-Q}
Q_{0}=\bigg(\bigcap\limits_{N\geq C_{0}} \bigcup\limits_{Q_j\in \F_{N}} Q_j\bigg)\cup \bigg(\bigcup\limits_{N\geq C_{0}} \Big(Q_{0}\setminus \bigcup\limits_{Q_j\in \F_{N}} Q_j\Big)\bigg)
=:
E_{0}\cup \bigg(\bigcup\limits_{N\geq C_{0}} E_N\bigg).
\end{equation}
The fact that $\sigma\ll \omega$ implies
\begin{equation}\label{e:E0}
\sigma(E_0)
\le
\sigma\big(\{x\in Q_0:\, \,  d\sigma/d\omega=0\, \, \, \mbox{or}\, \, \, d\sigma/d\omega=\infty\}\big)=0.
\end{equation}

The following proposition is the core result of this section.

\begin{proposition}\label{prop:CAD-d}
$\Omega_{\F_N,Q_0}$ is chord-arc domain for every $N\ge C_0$.
\end{proposition}

Observe that $E_N=Q_{0}\setminus \bigcup\limits_{Q_j\in \F_{N}} Q_j\subset \pom\cap\partial\Omega_{\F_N,Q_0}$ (cf. \eqref{prop61}). This, \eqref{eqn:cover-Q}, \eqref{e:E0} and the previous proposition give \eqref{maind} for the portion of the boundary corresponding to $Q_0$. Now we observe that $\pom=\bigcup_{Q\in\dd_{k_0}} Q$ and \eqref{maind} follows.

The proof of Proposition \ref{prop:CAD-d} being somewhat long, we break the argument into several steps. Fix any integer $N\ge C_0$. Let $\eta(N)$ be a sufficiently small constant depending on $N$ to be specified below. Recalling Definition \ref{def1.dyadcork} we set
\[
\mc{B}^{N}:=\left\{Q\in\dd_{Q_{0}}:\, \, Q \mbox{\ does not satisfy the $\eta(N)$-exterior Corkscrew condition}\right\}.
\]

Let us introduce some additional notation. For every $Q\in\dd_{Q_0}$, we set
\begin{equation}\label{eq2.5}
\alpha_Q := \left\{
\begin{array}{ll}
\sigma(Q)\,,&
\,\,{\rm if\,} Q\in\dd_{\F_{N},Q_{0}}\cap\mc{B}^{N},
\\[6pt]
0\,,&\,\,
{\rm otherwise}.
\end{array}
\right.
\end{equation}
For any subcollection $\dd'\subset\dd_{Q_{0}}$, we set
\begin{equation}\label{mut-def}
\mut(\dd'):= \sum_{Q\in\dd'}\alpha_{Q}.
\end{equation}
We shall see that the family $\dd_{\F_N, Q_0}\cap \mc{B}^N$ satisfies a packing condition with respect to the surface measure provided that $\eta(N)$ is small enough; that is, $\mut$ is a discrete
Carleson measure. In the argument that follows, we emphasize that constants are allowed to depend on $N$.

\begin{lemma}\label{lamma:pack}
Under the setup above, for each $N\ge C_0$ there exists $0<C_N<\infty$ (independent of $Q_0$) such that if $\eta(N)$ is small enough (depending on $N$ and the ADR and 1-sided NTA constants of $\Omega$), then  $\mut$ is a discrete
Carleson measure:
\begin{align}\label{sigmaQQ0}
\sup_{Q_0'\in \dd_{Q_0}} \frac{\mut(\dd_{Q_0'})}{\sigma(Q_0')}
=
\sup_{Q_0'\in \dd_{Q_0}} \frac1{\sigma(Q_0')}\sum_{\substack{Q\in\dd_{Q'_0}\\ Q\in\dd_{\F_{N},Q_{0}}\cap\mc{B}^{N}}} \sigma(Q)
\le C_N<\infty.
\end{align}
\end{lemma}

\begin{proof}
We first normalize the Green function $G(X_{0},\cdot)$ as we did with $\omega^{X_0}$. Set
\[
\mc{G}(Y):=C_0\,\sigma(Q_{0})\,G(X_{0},Y),
\]
where $X_{0}$ is the corkscrew point relative to $\kappa\,\Delta_{Q_0}^*$ as explained above and $C_{0}$ is the constant as in \eqref{normalizedomega*}. Note that our choice of $X_0$ guarantees that $\mc{G}\in W^{1,2}_0(2\,B_{Q_0}^*\cap\Omega)$ and $\mc{G}$ is harmonic in $2\,B_{Q_0}^*\cap\Omega$. Because all of our estimates below take place in $2\,B_{Q_0}^*\cap\Omega$ (since $T_{Q_0}\subset B_{Q_0}^*\cap\Omega$), this observation ensures the computations below are meaningful. Also, use of a Caffarelli-Fabes-Mortola-Salsa estimate and doubling of $\omega$ in $2\,\Delta_{Q_0}^*$ are legitimate under this regime. We note that in the harmonic case that we are currently considering these estimates have been proved when the domain is bounded in \cite{Aikawa}. Passing from bounded to unbounded requires a limiting argument along the lines in \cite[Section 3]{HM}. Further details will appear in the forthcoming paper  \cite{HMT-general}.

Fix a cube $Q\in \dd_{\F_{N},Q_{0}}\cap \mc{B}^{N}$ and a point $z_{Q}\in\Delta_{Q}\subset Q$. Set $B_Q':=B(z_{Q}, r_{Q}/4)$
and let $\phi_{Q}\in C^{\infty}_{0}(B_Q')$ with $0\leq \phi_{Q}\leq 1$, $\phi\equiv 1$ on $\frac12 B_Q'$, and $\|\nabla \phi_{Q}\|_{\infty}\lesssim r_{Q}^{-1}$, where $r_{Q}\approx \ell(Q)$. Then, from \eqref{omegasigmabdd} and \cite{HMT-general} (see also \cite{HM}), there exists a uniform constant $C_1>1$ (depending only on the ADR and 1-sided NTA constants of $\Omega$) such that
\begin{align}
\label{sigmaIandII}
(N\,C_1)^{-1} \sigma(Q)
&\le
C_1^{-1}\omega(Q)
\le \int\limits_{\partial\Omega}\phi_{Q}\,\rd\omega
=
-\iint\limits_{\Omega}\nabla \mc{G}\cdot\nabla\phi_{Q}\,\rd X\\ \nonumber
%=-\iint\limits_{\Omega}(\nabla \mc{G}-\vec{\alpha})\nabla \phi_{Q}\rd X + \iint\limits_{\Omega}\vec{\alpha}\nabla \phi_{Q}\rd X\\
&=
-\iint\limits_{\Omega}(\nabla \mc{G}-\vec{\alpha})\cdot\nabla \phi_{Q}\,\rd X
-
\iint\limits_{\ree}\vec{\alpha}\cdot\nabla \phi_{Q}\rd X+\iint\limits_{\Omega_{\rm ext}}\vec{\alpha}\cdot\nabla \phi_{Q}\,\rd X\\ \nonumber
&=
-\iint\limits_{\Omega}(\nabla \mc{G}-\vec{\alpha})\cdot\nabla \phi_{Q}\,\rd X
+
\iint\limits_{\Omega_{\rm ext}}\vec{\alpha}\cdot\nabla \phi_{Q}\,\rd X
\\ \nonumber
&=:
-\mathcal{I}+\mathcal{II}.
\end{align}
Here $\vec{\alpha}$ is a constant vector given by
\[
\vec{\alpha}:=\frac{1}{|U_{Q,\epsilon }|}\iint\limits_{U_{Q,\epsilon }} \nabla \mc{G}\rd X,
\]
where $\epsilon$ is a small constant depending on $N$ that we specify below and $U_{Q, \epsilon}:=\Omega_{\F_{N}(\epsilon\, r_Q),Q}$ is the geometric sawtooth region relative to $\F_{N}(\epsilon\, r_Q)$ defined in \S\ref{ss:grid}. Note that (see Figure \ref{UandSigma})
$$
(\Omega\setminus U_{Q,\epsilon})\cap B_Q'
\subset
\Sigma_\epsilon
:=
\{X\in\Omega:\, \, \delta(X)\lesssim \epsilon\,  \ell(Q)\}.
$$
\begin{figure}[!ht]
\def\firstcircle{(0,0) circle (4cm)}
\def\secondcircle{(0,0) circle (1cm)}
\centering
\begin{tikzpicture}[scale=.4]
\draw plot [smooth] coordinates {(-9.6,6)(-9.3,4)(-8.8,2.8)(-8.5,2.1)(-8,1)(-6,.6) (-4,.3)(-2,-.3) (0,0)(2,.3)(4,-.3)(6,-.6)(8,0)(8.5,2.1)(8.8,2.8)(9.3,4)(9.6,6)(3,6.3)(-3,5.8)(-9.6,6)};
\node at (5.5,4.5) {$U_{Q,\epsilon}=\Omega_{\F_{N}(\epsilon r_Q),Q}$};
%\draw[rotate=180, yshift=2cm] plot [smooth] coordinates {(-9.6,6)(-9.3,4)(-8.8,2.8)(-8.5,2.1)(-8,1)(-6,.6) (-4,.3)(-2,-.3) (0,0)(2,.3)(4,-.3)(6,-.6)(8,0)(8.5,2.1)(8.8,2.8)(9.3,4)(9.6,6)(3,6.3)(-3,5.8)(-9.6,6)};
\draw[dashed, yshift=-1cm, red] plot [smooth] coordinates {(-8,1)(-6,.6) (-4,.3) (-2,-.3) (0,0)(2,.3)(4,-.3)(6,-.6)(8,-1)};
\filldraw (0,-1) circle (2pt) node[below] {$z_{Q}$};
\node[left, red] at (-8,0) {$\partial\Omega$};
\draw \firstcircle;
\begin{scope}
\clip \firstcircle;
\draw[dashed, fill=gray] plot [smooth] coordinates {(-8,1)(-6,.6) (-4,.3) (-2,-.3) (0,0)(2,.3)(4,-.3)(6,-.6)(8,-1)} (8,-1)--(8,4)--(-8,4)--(-8,1);
\draw[dashed, yshift=-1cm, fill=gray, opacity=0.5] plot [smooth] coordinates {(-8,1)(-6,.6) (-4,.3) (-2,-.3) (0,0)(2,.3)(4,-.3)(6,-.6)(8,-1)} (8,-1)--(8,5)--(-8,5)--(-8,1);
\end{scope}
\draw[dashed] plot [smooth] coordinates {(-6,.6) (-4,.3) (-2,-.3) (0,0)(2,.3)(4,-.3)(6,-.6)};
\draw[dashed, yshift=-1cm] plot [smooth] coordinates {(-6,.6) (-4,.3) (-2,-.3) (0,0)(2,.3)(4,-.3)(6,-.6)};
\node at (5,-3.2) {$\Sigma_{\epsilon}\cap 2\,B_Q'$};
%\node at (0,2) {$U_{Q, \epsilon }$};
\draw[->]  (2,-.3) edge[bend left=10] (4.6,-2.6);
\draw[<->] (6,-.6)--(6,-1.6);
\node[right] at (6.4,-1.1) {$\lesssim\epsilon\, \ell(Q)$};
\node[below] at (0,-4) {$2\,B_Q'$};
\end{tikzpicture}
\caption{$U_{Q, \epsilon }$ and $\Sigma_{\epsilon}$.}
\label{UandSigma}
\end{figure}
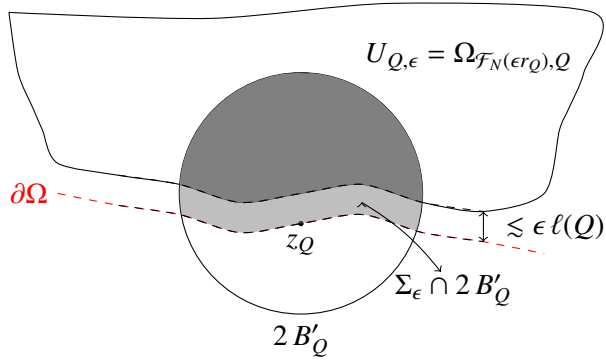

\noindent Next, observe that for every $X\in U_{Q,\epsilon}$ one has $\epsilon\,\ell(Q) \lesssim\delta(X)\lesssim \dist(X,\Delta_{Q})\lesssim \ell(Q)$.
Moreover, $|U_{Q, \epsilon}|\gtrsim \ell(Q)^{n+1}$ with implicit constant independent of the number $\epsilon$. Indeed, since $Q\in\dd_{\F_N,Q_0}$, we have that $Q\in\dd_{\F_N(\epsilon\,r_Q),Q}$ provided $\epsilon$ is small enough. Also, since $\W_Q\neq\emptyset$, there is $I\in\W_Q$ such that $\ell(I)\approx\ell(Q)$ and $\dist(I,Q)\lesssim \ell(Q)$. In particular $I\subset \Omega_{\F_{N}(\epsilon\, r_Q),Q}=U_{Q,\epsilon}$. Therefore,
\begin{align}
\label{IinUQepsilon}
|U_{Q,\epsilon }|\geq |I|\approx \ell(Q)^{n+1},
\end{align}
and this estimate does not depend on $\epsilon$ (provided $\epsilon$ is small enough).

We now show $|\vec{\alpha}|\leq C_{N}$.  To this end, first observe that $U_{Q,\epsilon }\subset T_{Q}$, where $T_{Q}$ is the Carleson box relative to $Q$. Using this observation, Caccioppoli's inequality, Harnack's inequality, a Caffarelli-Fabes-Mortola-Salsa estimate, and doubling of $\omega$  (see \cite{Aikawa} in the bounded case or \cite{HMT-general} in general),  we obtain
\begin{align}\label{I2}
|\vec{\alpha}\,|
&\lesssim
\ell(Q)^{-(n+1)}\iint_{T_Q} |\nabla \mc{G}|\,dX
\\ \nonumber
&
\le
\ell(Q)^{-(n+1)}\sum_{Q'\in\dd_Q} \sum_{I\in \W_{Q'}^*} |I|^{\frac12}\Big(\iint_{I^*} |\nabla \mc{G}|^2\,dX\Big)^{\frac12}
\\ \nonumber
&
\lesssim
\ell(Q)^{-(n+1)}\sum_{Q'\in\dd_Q} \sum_{I\in \W_{Q'}^*} |I| \frac{\mc{G}(X(I))}{\delta(X(I))}
\\ \nonumber
&
\approx
\ell(Q)^{-(n+1)}\sum_{Q'\in\dd_Q} \sum_{I\in \W_{Q'}^*} |I| \frac{\omega(Q')}{\sigma(Q')}
\\ \nonumber
&\lesssim
\ell(Q)^{-(n+1)}\sum_{Q'\in\dd_Q} \omega (Q')\,\ell(Q')
\\ \nonumber
&
=
\ell(Q)^{-n}\sum_{k=0}^\infty 2^{-k}\sum_{\substack{Q'\in\dd_Q\\ \ell(Q')=2^{-k}\,\ell(Q)}} \omega(Q')
\\ \nonumber
&
\lesssim
\frac{\omega(Q)}{\sigma(Q)}
\le N.
\end{align}
Note that the last estimate follows from \eqref{omegasigmabdd}, because $Q\in\dd_{\F_N,Q_0}$ and there is no dependence on $\epsilon$.

We are now ready to estimate $\mathcal{II}$ in \eqref{sigmaIandII}.  Recall that $Q\in \mc{B}^{N}$. By \cite[Lemma 5.7]{HM}, failure of the $\eta(N)$-exterior Corkscrew property implies that $|\Omega_{\rm ext}\cap B_Q'|\lesssim\eta(N)\,r_Q^{n+1}$. This and  \eqref{I2} give
\begin{equation}\label{sigmaII}
|\mathcal{II}|
\lesssim
|\vec{\alpha}\,|\,r_Q^{-1}\,|\Omega_{\rm ext}\cap B_Q'|
\lesssim
N\,\eta(N)\, r_Q^{n}
\approx
N\,\eta(N)\, \sigma(Q)
<
\frac1{4\,N\,C_1}\,\sigma(Q),
\end{equation}
where in the last estimate we have chosen $\eta(N)$ sufficiently small ($\eta(N)\le (N^2\,M)^{-1}$ with $M\gg 1$ depending only on the ADR and 1-sided NTA constants).

We next estimate $\mathcal{I}$. To start,
\begin{multline}\label{dede}
|\mathcal{I}|
\lesssim
r_Q^{-1}\,\Big(
\iint\limits_{(\Omega\setminus U_{Q,\epsilon})\cap B_Q'} |\nabla \mc{G}-\vec{\alpha}|\,\rd X
+
\iint\limits_{U_{Q, \epsilon }} |\nabla \mc{G}-\vec{\alpha}| \, \rd X\Big)\\
\\
\lesssim
\ell(Q)^{-1}\,\Big( \iint\limits_{B_Q'\cap \Sigma_\epsilon} |\nabla \mc{G}-\vec{\alpha}|\,\rd X
+
\iint\limits_{U_{Q, \epsilon }} |\nabla \mc{G}-\vec{\alpha}| \, \rd X\Big)
=:\ell(Q)^{-1}\big(\mathcal{I}_1+\mathcal{I}_{2}\big).
\end{multline}
Using \cite[Lemma 5.3]{HM}, we obtain
\begin{align}
\label{sigmaI1}
\mathcal{I}_1\lesssim
|\vec{\alpha}|\,|B_Q'\cap \Sigma_\epsilon|+\iint\limits_{B_Q'\cap \Sigma_\epsilon} |\nabla \mc{G}|\,\rd X
\lesssim
N\,\epsilon\,\ell(Q)^{n+1}+ \mathcal{I}_3
\lesssim
N\,\epsilon\,\ell(Q)\,\sigma(Q)+\mathcal{I}_3.
\end{align}
We estimate $\mathcal{I}_3$, as follows. Given $I\in\W$, let $Q_I^*$ denote one of its nearest cubes with $\ell(Q_I^*)=\ell(I)$. Using the same ideas as in \eqref{I2},
\begin{multline*}
\mathcal{I}_3
\le
\sum_{\substack{I\in \W: I\cap B_Q'\neq\tinyemptyset\\ \ell(I)\lesssim \epsilon\,\ell(Q)}} \iint_{I} |\nabla \mc{G}|\,\rd X
\le
\sum_{\substack{I\in \W: I\cap B_Q'\neq\tinyemptyset\\ \ell(I)\lesssim \epsilon\,\ell(Q)}}
|I|^{\frac12}\Big(\iint_{I}  |\nabla \mc{G}|^2\,\rd X\Big)^{\frac12}
\\ \nonumber
\lesssim
\sum_{\substack{I\in \W: I\cap B_Q'\neq\tinyemptyset\\ \ell(I)\lesssim \epsilon\,\ell(Q)}}  |I| \frac{\mc{G}(X(I))}{\delta(X(I))}
\approx
\sum_{\substack{I\in \W: I\cap B_Q'\neq\tinyemptyset\\ \ell(I)\lesssim \epsilon\,\ell(Q)}}  |I| \frac{\omega(Q_I^*)}{\sigma(Q_I^*)}
\\ \nonumber
\lesssim
\sum_{\substack{I\in \W: I\cap B_Q'\neq\tinyemptyset\\ \ell(I)\lesssim \epsilon\,\ell(Q)}}  \omega(Q_I^*)\, \ell(I)
=
\sum_{k: 2^{-k}\lesssim \epsilon\,\ell(Q)}2^{-k}
\sum_{\substack{I\in \W: I^*\cap B'_Q\neq\tinyemptyset\\ \ell(I)=2^{-k}}}  \omega(Q_I^*).
\end{multline*}
Note that if $k$ is fixed, then the family $\{Q_I^*\}_{I\in\W: \ell(I)=2^{-k}}$ has bounded overlap. Also, if $I$ meets $B_Q'$, then $\ell(I)\lesssim r_Q\approx\ell(Q)$. Hence $Q_I^*\subset C\,\Delta_{Q}'=C\,B_Q'\cap\pom$ for some uniform constant $C$. Thus,
\begin{align}\label{est-I3}
\mathcal{I}_3
\lesssim
\omega(C\,\Delta'_Q)\sum_{k: 2^{-k}\lesssim \epsilon\,\ell(Q)}2^{-k}
\lesssim
\omega(Q) \epsilon\,\ell(Q)
\lesssim
N\,\epsilon\,\ell(Q)\,\sigma(Q),
\end{align}
where we used doubling of $\omega$  (see \cite{Aikawa} in the bounded case or \cite{HMT-general} in general) and Harnack's inequality; moreover, in the last estimate we invoked \eqref{omegasigmabdd}, since $Q\in\dd_{\F_N,Q_0}$. Gathering \eqref{dede}, \eqref{sigmaI1}, and \eqref{est-I3}, we obtain
\begin{align}
\label{I1final}
|\mathcal{I}|
\lesssim
N\,\epsilon\,\,\sigma(Q)+\ell(Q)^{-1}\,\mathcal{I}_2
<
\frac1{4\,N\,C_1}\,\sigma(Q)+\ell(Q)^{-1}\,\mathcal{I}_2
\end{align}
provided we choose $\epsilon$ sufficiently small ($\epsilon\le (N^2\,M)^{-1}$ with $M\gg 1$ depending only on the ADR and 1-sided NTA constants will suffice). For later use, we assume that $\epsilon=2^{-K_\epsilon}$ for some $K_\epsilon\in \N$.

From \eqref{sigmaIandII}, \eqref{sigmaII}, and \eqref{I1final}, it follows  that
\begin{align}\label{sigmaQ1412}
(C_1\,N)^{-1}\,\sigma(Q)
\leq
|\mathcal{I}|+|\mathcal{II}| \leq \frac{1}{2\,N\,C_1}\sigma(Q)+\ell(Q)^{-1}\,\mathcal{I}_2.
\end{align} Upon rearranging the inequality, we conclude that
$$
\sigma(Q)
\lesssim
2\,C_1\,N\,\ell(Q)^{-1}\,\mathcal{I}_2.
$$
Recall that at this point $\eta(N)$ and $\epsilon=\epsilon(N)$ are fixed and depend on $N$ and the ADR and 1-sided NTA constants of $\Omega$.

To continue with the previous estimate, we again use Harnack's inequality, a Caffarelli-Fabes-Mortola-Salsa estimate and that $\omega$ is doubling (see \cite{Aikawa} in the bounded case or \cite{HMT-general} in general):
\begin{equation}
\label{ferfrfR}
\frac{\mc{G}(X)}{\ell(Q)}
\approx_N
\frac{\mc{G}(X)}{\delta(X)}
\approx_N
\frac{\omega(Q)}{\sigma(Q)}
\approx_N 1,
\qquad \forall\, X\in U_{Q,\epsilon}
.
\end{equation}
Next we need a sharper version of a Poincar\'e inequality from \cite[Lemma 4.8]{HM}. 
In that reference such estimate takes place on the set $U_{Q,\epsilon}$, but its right hand  has a slight fattening of the Whitney regions. However, a careful examination of the proof 
of \cite[Lemma 4.8]{HM} reveals that one can obtain the Poincare inequality without 
fattening the Whitney regions since they are comprised of Whitney cubes.  Details of the latter approach will appear in  \cite{HMT}. This and \eqref{ferfrfR} give
\begin{multline}\label{trfrg}
\sigma(Q)
\lesssim_N
\ell(Q)^{-1}\,|U_{Q,\epsilon}|^{\frac12}\,\Big(\iint\limits_{U_{Q, \epsilon }} |\nabla \mc{G}-\vec{\alpha}|^2 \, \rd X\Big)^{\frac12}
\\
\lesssim_{N}
\ell(Q)^{\frac{n+1}2}\,
\Big(\iint\limits_{U_{Q, \epsilon }} |\nabla^2 \mc{G}(X)|^2 \, \rd X\Big)^{\frac12}
\approx_N
\sigma(Q)^{\frac12}\,\Big(\iint\limits_{U_{Q, \epsilon }} |\nabla^2 \mc{G}(X)|^2 \,\mc{G}(X) \rd X\Big)^{\frac12}.
\end{multline}
Hiding this time $\sigma(Q)^{\frac12}$, we conclude that
\begin{align}
\label{sigmaQGreen}
\sigma(Q) & \lesssim_{N} \iint\limits_{U_{Q, \epsilon }} |\nabla^{2} \mc{G}|^{2}\,\mc{G}\, \rd X,
\qquad
\forall\,
Q\in \dd_{\F_N, Q_0}\cap \mc{B}^N.
\end{align}

Recall that our goal is to obtain \eqref{sigmaQQ0}. Fix $Q'_{0}\subset \dd_{Q_{0}}$. We may assume that $Q_0'\in\dd_{\F_{N},Q_{0}}$ (otherwise $\mut(\dd_{Q_0'})=0$ and the desired estimate follows), in which case we have $Q_j\subsetneq Q_0'$ whenever $Q_0'\cap Q_j\neq\emptyset$. Hence
\[
\dd_{\F_{N}',Q_{0}'}=\dd_{Q_{0}'}\cap \dd_{\F_{N},Q_{0}},
\qquad\mbox{where}\qquad
\F_{N}':=\{Q_{j}\in\F_{N}:\, \, Q_{j}\subsetneq Q_{0}'\}.
\]
Recall that we chose $\epsilon$ to be of the form $\epsilon =2^{-K_{\epsilon}}$ for some $K_\epsilon\in\N$. Let
\[
\F^{\star}_{N}:=\bigcup\limits_{Q\in\F'_{N}}\big\{Q'\in\dd_{Q}:\, \, \ell(Q')= 2^{-K_{\epsilon }}\ell(Q)=\epsilon \ell(Q)\big\}.
\]
Note that $\F^{\star}_{N}\subset\dd_{Q_{0}'}$ and it is a disjoint family.
For ease of notation, we let
$$
\Omega^{\star}
:=
\Omega_{\dd_{\F_{N}^{\star},Q_{0}'}}^*
:=
\interior\Big(\bigcup\limits_{Q\in \dd_{\F_{N}^{\star},Q_{0}'}}U_{Q}^*\Big)
:=
\interior\Big(\bigcup\limits_{Q\in \dd_{\F_{N}^{\star},Q_{0}'}}\bigcup_{I\in W_Q^*} I^{**}\Big),
$$
where $I^{**}=(1+2\,\lambda)\,I$.   Thus,
\begin{multline}
\label{unionUqepsilon}
\bigcup\limits_{Q\in \dd_{\F_{N}',Q_{0}'}}U_{Q,\epsilon } \subset \bigcup\limits_{Q\in \dd_{\F_{N}',Q_{0}'}}\, \bigcup\limits_{\substack{Q'\in\dd_{Q} \\ \epsilon \ell(Q)<\ell(Q')\leq \ell(Q)}}U_{Q}
\subset\bigcup\limits_{Q\in \dd_{\F_{N}^{\star},Q_{0}'}}U_{Q}
\\
\subset
\interior\Big(\bigcup\limits_{Q\in \dd_{\F_{N}^{\star},Q_{0}'}}U_{Q}^*\Big)=
\Omega^{\star}.
\end{multline}
This, the fact that the family $\{U_{Q,\epsilon}\}_{Q\in\dd}$ has have bounded overlap (depending on $\epsilon$ and hence on $N$), see \cite{AHMNT} or \cite{HMT}, and \eqref{sigmaQGreen} yield
\begin{align}
\label{OmegastarFstarN}
\mut(\dd_{Q_0'})
\le
\sum_{\substack{Q\in\dd_{Q'_0}\\ Q\in\dd_{\F_{N},Q_{0}}\cap\mc{B}^{N}}} \sigma(Q)
\lesssim_{N}
\sum\limits_{Q\in\dd_{\F_{N}',Q_{0}'}}\iint\limits_{U_{Q,\epsilon }} |\nabla^{2} \mc{G}|^{2}\mc{G}\rd X
\lesssim_{N} \iint\limits_{\Omega^{\star}} |\nabla^{2} \mc{G}|^{2}\mc{G}\rd X.
\end{align}

We now claim that
\begin{align}\label{frefr}
\frac{\omega(Q)}{\sigma(Q)}\approx_N 1,
\qquad \forall\,Q\in \dd_{\F_{N}^{\star},Q_{0}'}.
\end{align}
This is clear if $Q\in \dd_{\F'_{N},Q_{0}'}\subset \dd_{\F_{N},Q_{0}}$  by \eqref{omegasigmabdd}. Suppose next that $Q\in \dd_{\F_{N}^{\star},Q_{0}'}\setminus\dd_{\F'_{N},Q_{0}'}$. Then there is $Q_j\in \F_N'$ such that $Q\subset Q_j$. We can split $Q_j$ into its $\epsilon$-descendants (recall that $\epsilon=2^{-K_\epsilon}$) and we can find $Q_j'\in \F_{N}^{\star}$ such that $Q_j'\cap Q\neq\emptyset$. In turn, since $Q\in \dd_{\F_{N}^{\star},Q_{0}'}$, necessarily $Q_j'\subsetneq Q\subset Q_j$ with $\ell(Q_j')=\epsilon\,\ell(Q_j)$. Using this, ADR, doubling of $\omega$, and \eqref{omegasigmabdd} which clearly holds for the father $\widehat{Q}_j$ of $Q_j$, we conclude that $$
\frac{\omega(Q)}{\sigma(Q)}\approx_N \frac{\omega(\widehat{Q}_j)}{\sigma(\widehat{Q}_j)}\approx_N 1,
$$ as desired.

In order to integrate by parts, we need to get away from the boundary. We would like to introduce a new domain as in Figure \ref{integraionbypartsfigure} but in a way that the new domain has ADR boundary with bounds that are independend of $M$. To do so, we introduce a large parameter $M$ and define $\F_{N,M}^{\star}=\F_{N}^{\star}(2^{-M}\,\ell(Q_0'))$; that is, $\F_{N,M}^{\star}\subset\dd_{Q_0'}$ is the family of maximal cubes of the collection $\F_{N}^{\star}$ augmented to include all dyadic cubes of size smaller than or equal to $2^{-M}\,\ell(Q_0)$. In particular, $Q\in\dd_{\F_{N,M}^{\star}, Q_0'}$ if and only if $Q\in \dd_{\F_{N}^{\star},Q_0'}$ and $\ell(Q)>2^{-M}\,\ell(Q_0')$. Clearly,
$\dd_{\F_{N,M}^{\star}, Q_0'}\subset \dd_{\F_{N,M'}^{\star}, Q_0'}$ if $M\le M'$, and therefore, $\Omega^\star_M:=\Omega_{\F_{N,M}^{\star} ,Q_0'}^*\subset \Omega_{\F_{N,M'}^{\star} ,Q_0'}^*\subset \Omega_{\dd_{\F_{N}^{\star} ,Q_0'}}^*=\Omega^\star$. This and the monotone convergence theorem give
\begin{align}\label{IBP-limit}
\iint_{\Omega^\star} |\nabla^2 \mc{G}|^2\,\mc{G}\,dX
=
\lim_{M\to\infty} \iint_{\Omega^\star_M} |\nabla^2 \mc{G}|^2\,\mc{G}\,dX.
\end{align}
Thus, we may bound each of the right hand terms in \eqref{OmegastarFstarN} with bounds that are uniform in $M$ using integration by parts. 
\def\firstrectangle{(-5,-5) rectangle (8,.7)}
\begin{figure}[!ht]
\centering
\begin{tikzpicture}[scale=0.8]
\node[above] at (5,2) {$\Omega^{\star}_M$};
\draw[thick, fill=gray, opacity=0.6, name path=Sawtooth1] plot[smooth] coordinates {(-5,3)(-4,0)(-3,-1)(-2,.5)(-1,-.5)(0,1) (1,2) (2,1) (3,1.5) (5,-1) (6,0)(7,3)(0,2.9)(-5,3)};
\draw plot[smooth] coordinates {(-5,-2)(-3,-1)(0,-.4)(5,-1)(7,-2)};
\draw[white, dashed, name path=L1] (-5,.7)--(7,.7);
%\draw[dashed, red, opacity=0.5, name path=W1] (-5,1)--(7,1);
\clip \firstrectangle;
\draw[thick, fill=gray, opacity=0.6] plot[smooth] coordinates {(-5,3)(-4,0)(-3,-1)(-2,.5)(-1,-.5)(0,1) (1,2) (2,1) (3,1.5) (5,-1) (6,0)(7,3)(0,2.9)(-5,3)};
\path [name intersections={of=Sawtooth1 and L1, by={a,b,c,d}}];
\draw[red, dashed] (a)--(b);
\draw[red, dashed] (c)--(d);
%\draw (a)--(b);
\node[below] at (7,-2) {$\partial\Omega$};

\draw[<->] (-0.5,-.4)--(-0.5,.69);
\node[right] at (-0.48,0.15) {\footnotesize$\approx 2^{-M}$};

\end{tikzpicture}
\null\vskip-2.5cm\null
\caption{The domain $\Omega^{\star}_M$ where we do integration by parts.}
\label{integraionbypartsfigure}
\end{figure}
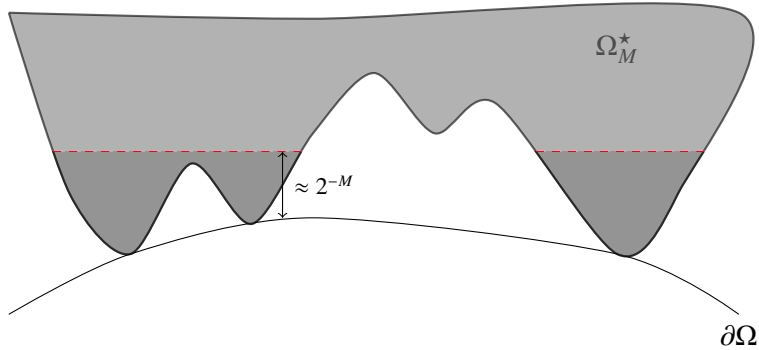

Fix $M$ large. Note that $\Omega^{\star}_M\subset \overline{\Omega^{\star}_M}\subset 2\,B_{Q_0^*}\cap\Omega$, and therefore, we again have the needed PDE properties at our disposal. Let ``$\partial$'' denote a fixed generic derivative. Easy calculations show that in $2\,B_{Q_0^*}\cap\Omega$ we can use that $\mc{G}$ is harmonic and then
\[
\Delta((\partial \mc{G})^{2})
=
2\,\div[(\partial \mc{G})\,\nabla(\partial \mc{G})]=2\,|\nabla(\partial \mc{G})|^{2}
\]
 and
\[
\Delta((\partial \mc{G})^{2})\,\mc{G}
=
\div[\nabla((\partial\mc{G})^{2})\,\mc{G}] -\nabla((\partial\mc{G})^{2})\cdot\nabla \mc{G}
=
\div[\nabla ((\partial\mc{G})^{2})\,\mc{G}] -\div[(\partial\mc{G})^{2}\nabla \mc{G}].
\]
Since the domain $\Omega^\star_M$ is comprised of a finite union of fattened Whitney cubes, its boundary consists of portions of faces of those cubes. Thus, its (outward) unit normal $\nu$ is well defined a.e.~on $\partial\Omega^\star_M$ and the divergence theorem can be applied.
Hence
\begin{align}
\label{intbypartsomegstar}
2\,\iint\limits_{\Omega^{\star}_M} |\nabla^{2} \mc{G}|^{2}\,\mc{G}\rd X
&=
\iint\limits_{\Omega^{\star}_M}\Delta((\partial \mc{G})^{2})\,\mc{G}\rd X
\\ \nonumber
&=
\int\limits_{\partial\Omega^{\star}_M}\big[\nabla( (\partial\mc{G})^{2})\,\mc{G}-(\partial\mc{G})^{2}\,\nabla \mc{G}\big]\cdot\nu\,
\rd \rh^n\big|_{\partial\Omega^{\star}_M}
\\ \nonumber
&\lesssim
\int\limits_{\partial\Omega^{\star}_M} \big[ |\nabla^2 \mc{G} |\,|\nabla \mc{G}|\,\mc{G}+|\nabla \mc{G}|^3\big]\,\rd \rh^n\big|_{\partial\Omega^{\star}_M}
\\ \nonumber
&\lesssim
\int\limits_{\partial\Omega^{\star}_M} \Big(\frac{\mc{G}}{\delta}\Big)^3\,\rd \rh^n\big|_{\partial\Omega^{\star}_M},
\end{align}
where in the last inequality we have used standard interior estimates for harmonic functions. Note that for every $X\in \partial\Omega^{\star}_M\subset\Omega$ we have that there exists $I\in\W$ such that $X\in\partial I^{**}$ with $I\in \W_Q^*$ and $Q\in\dd_{\F_{N,M}^{\star}, Q_0'}
\subset \dd_{\F_{N}^{\star}, Q_0'}$. Hence, by Harnack's inequality, the Caffarelli-Fabes-Mortola-Salsa estimate, the fact that $\omega$ is doubling (see \cite{Aikawa} in the bounded case or \cite{HMT-general} in general) and \eqref{frefr}, we conclude that
$$
\frac{\mc{G}(X)}{\delta(X)}
\lesssim
\frac{\mc{G}(X(I))}{\ell(I)}
\approx
\frac{\omega(Q)}{\sigma(Q)}
\lesssim_N
1.
$$
Plugging this into \eqref{intbypartsomegstar} and using that $\partial\Omega^{\star}_M$ is ADR (since it is a sawtooth domain) with bounds that are uniform in $N$ and $M$ (see \cite[Lema 3.61]{HM}), we conclude that
$$
\iint\limits_{\Omega^{\star}_M} |\nabla^{2} \mc{G}|^{2}\,\mc{G}\rd X
\lesssim_N
\rh^n(\partial\Omega^{\star}_M)
\approx
\diam(\partial\Omega^{\star}_M)^n
\lesssim
\ell(Q_0')^n
\approx
\sigma(Q_0').
$$
Combining this with \eqref{OmegastarFstarN} and \eqref{IBP-limit} it follows that $\mut(\dd_{Q_0'})\lesssim_N\sigma(Q_0')$, as desired. This completes the proof of \eqref{sigmaQQ0}.
\end{proof}

\medskip

Equipped with Lemma \ref{lamma:pack}, we immediately see that for every $Q_{0}'\in \dd_{\F_{N}, Q_{0}}$ there exists $Q_0''\in\dd_{Q_0'}$ such that
\begin{align}
\label{Q0''notinFNBN}
2^{-C_{N}}\ell(Q_{0}')\leq \ell(Q_0'')\leq \ell(Q_{0}')\qquad \mbox{and}\qquad Q_{0}''\notin  \dd_{\F_{N}, Q_{0}}\cap \mc{B}^{N},
\end{align}
where $C_N$ is the constant in \eqref{sigmaQQ0}. Otherwise,
\begin{align*}
(C_{N}+1)\,\sigma(Q_{0}')=\sum\limits_{\substack{Q\in\mathbb{D}_{Q_{0}'}\\ 2^{-C_{N}}\ell(Q_{0}')\leq \ell(Q)\leq \ell(Q_{0}')}} \sigma(Q)
\leq
\sum\limits_{\substack{Q\in\dd_{Q_{0}'}\\Q\in\mathbb{B}^{N}\cap \dd_{\F_{N},Q_{0}}}} \sigma(Q)\leq C_{N}\sigma(Q_{0}'),
\end{align*} which is absurd.

We now claim that
\begin{equation}\label{case1andcase2}
\left\{\begin{array}{ll} \mbox{For all }Q'_0\in \dd_{\F_N, Q_{0}},\mbox{ there exists }\widetilde{Q}_{0}'\in\dd_{Q_{0}'}\mbox{ such that }\\[0.1cm] 2^{-C_{N}}\le \frac{\ell(\widetilde{Q}_0')}{\ell(Q'_{0})}\le 1
\mbox{\ and either \ } \widetilde{Q}_{0}'\in\F_{N} \mbox{\ or\ }
\widetilde{Q}_{0}'\notin \mc{B}^{N}.\end{array}\right.
\end{equation}
To verify this claim, fix $Q_{0}'\in \dd_{\F_{N}, Q_{0}}$ and let $Q_0''\in\dd_{Q_0'}$ be a cube satisfying \eqref{Q0''notinFNBN}. We consider two separate possibilities.

\noindent
{\bf Case 1:} Suppose that $Q''_{0}\notin \dd_{\F_{N}, Q_{0}}$. Then  there exists $Q_{j}\in\F_{N}$ such that $Q''_{0}\subset Q_{j}$. As a consequence, $Q_0''\subset Q_j\cap Q_0$ with $Q'_{0}\in \dd_{\F_{N}, Q_{0}}$ give $Q_{j}\subsetneq Q_{0}'$. Thus, by \eqref{Q0''notinFNBN},
\[
Q''_{0}\subset Q_{j}\subsetneq Q_{0}'\qquad\mbox{and}\qquad 2^{-C_{N}}\ell(Q'_{0})\leq \ell(Q''_{0})\leq \ell(Q_{j})\leq \ell(Q_{0}').
\]
In this case, we pick $\widetilde{Q}_{0}':=Q_{j}$.

\noindent{\bf Case 2:} Suppose $Q_{0}''\notin \mc{B}^{N}$. In this case, we let $\widetilde{Q}_{0}':=Q_{0}''$ and the desired properties follow at once from \eqref{Q0''notinFNBN}.

We now have now all the ingredients required to prove Proposition \ref{prop:CAD-d}, and thus, complete the proof of \eqref{mainc} implies \eqref{maind}.

\begin{proof}[Proof of Proposition \ref{prop:CAD-d}]
By \cite[Lemma 3.61]{HM}, $\Omega_{\F_{N},Q_{0}}$ is a 1-sided NTA with ADR boundary. Hence all that remains is to show that $\Omega_{\F_{N},Q_{0}}$ satisfies the exterior corkscrew condition with constant depending on $N$. Fix any point $x\in\partial \Omega_{\F_{N},Q_{0}}$ and $0<r<\diam(Q_{0})\approx\diam(\Omega_{\F_{N},Q_{0}})$. There are two cases.

\noindent{\bf Case 1:} Suppose that $x\in\partial\Omega_{\F_{N},Q_{0}}$ with $0\le \delta(x)\le  r/M$, where $M$ is large enough, to be chosen. We first note that there exists $Q\in\dd_{Q_0}$ with $\ell(Q)\approx r/M$ such that $|x-x_Q|\lesssim r/M$. On the one hand, if $x\in\partial\Omega_{\F_{N},Q_{0}}\cap \partial\Omega$, then $x\in \overline{Q_0}$ by \eqref{prop61} and we can find $Q\in \dd_{Q_0}$ with $\ell(Q)\approx r/M$ and $x\in\overline{Q}$. On the other hand, if   $x\in\partial\Omega_{\F_{N},Q_{0}}\cap\Omega$, then by the definition of the sawtooth region, $x\in\partial I^*$, where $I\in\W_{Q'}^*$, $Q'\in\dd_{\F_N,Q_0}$, and $|x-x_{Q'}|\approx\dist(I,Q')\approx\ell(Q') \approx \ell(I)\approx\delta(x)\le r/M$.
Let us now take $Q\in\dd_{Q_0}$, an ancestor of $Q'$, such that $\ell(Q)\approx r/M$ and hence $|x-x_Q|\le |x-x_{Q'}|+|x_{Q'}-x_Q|\lesssim r/M$.

The proof now splits into two subcases.

\noindent{\bf Case 1a:} Suppose that $Q\notin \dd_{\F_{N}, Q_{0}}$. Then  $Q\subset Q_{j}$ for some $Q_j\in\F_{N}$ and by \cite[Lemma 5.9]{HM},
there exists a ball $B'\subset\ree\setminus \Omega_{\F_{N},Q_{0}}$ with center $x_{Q}$ and radius $r'\approx \ell(Q)\approx r/M$ such that $B'\cap\partial\Omega\subset Q$. If $y\in B'$, then we can guarantee
\[
|y-x|\leq |y-x_{Q}|+|x_{Q}-x|\lesssim r'+\frac{\ell(Q)}{M}
\lesssim
\frac{\ell(Q)}{M}<r
\]
by choosing $M$ sufficiently large. Hence $B'\subset B(x,r)$ and $x_{Q}$ is an exterior cork\-screw point relative to $B(x,r)\cap\partial \Omega_{\F_{N},Q_{0}}$ with corkscrew constant on the order of $M^{-1}$.

\noindent{\bf Case 1b:} Suppose that $Q\in \dd_{\F_{N}, Q_{0}}$. By \eqref{case1andcase2}, there exists $\widetilde{Q}\in\dd_{Q}$ such that
\begin{align}\label{aaa}
 2^{-C_{N}}\le \frac{\ell(\widetilde{Q})}{\ell(Q)}\le 1\qquad\mbox{and}\qquad\mbox{either }\widetilde{Q}\in\F_{N}\mbox{\ \  or\ \  }\widetilde{Q}\notin \mc{B}^{N}.
\end{align}
If $\widetilde{Q}\in\F_{N}$, we argue as in {\bf Case 1a}: \cite[Lemma 5.9]{HM} gives a ball $B'\subset\ree\setminus \Omega_{\F_{N},Q_{0}}$ with center  $x_{\widetilde{Q}}$ and radius $r'\approx\ell(\widetilde{Q})$. Once again, $B'\subset B(x, C\,r/M)\subset B(x,r)$ if $M$ is large enough. Hence $x_{Q}$ is an exterior corkscrew point relative to $B(x,r)\cap\partial \Omega_{\F_{N},Q_{0}}$ with corkscrew constant on the order of $2^{-C_N}\,M^{-1}$. Otherwise, if $\widetilde{Q}\notin \mc{B}^{N}$, then the definition of $\mc{B}^{N}$ yields $z_{\widetilde{Q}}\in \Delta_{\widetilde{Q}}\subset \widetilde{Q}$ and $X_{\widetilde{Q}}^-$ for which $B(X_{\widetilde{Q}}^-,\eta(N)\,\ell(\widetilde{Q}))\subset B(z_{\widetilde{Q}}, r_{\widetilde{Q}}/4)\cap\Omega_{\rm ext}$. Note that $B(z_{\widetilde{Q}}, r_{\widetilde{Q}}/4)\subset B(x,C\,r/M)\subset B(x,r)$ provided $M$ is large enough. Hence $X_{\widetilde{Q}}^-$ is an exterior corkscrew point relative to $B(x,r)\cap\partial \Omega_{\F_{N},Q_{0}}$ with corkscrew constant of the order of $2^{-C_N}\,M^{-1}$.

\noindent{\bf Case 2:} Suppose that $x\in\partial\Omega_{\F_{N},Q_{0}}$ with $\delta(x)> r/M$. In particular, $x\in\Omega$, and by the definition of the sawtooth region, $x\in\partial I^*\cap J$ where $I\in\W_Q^*$, $Q\in\dd_{\F_N,Q_0}$, $J\in\W$, and $\tau\,J\subset\Omega\setminus\Omega_{\F_{N},Q_{0}}$ for some $\tau\in (1/2,1)$. Note that $\ell(I)\approx\ell(J)\approx\delta(x)>r/M$. Hence we can easily find an exterior corkscrew point in the segment joining $x$ with the center of $J$ and the corkscrew constant will depend only on $M$ and the implicit constants in the previous estimates.

This completes the proof that $\Omega_{\F_{N},Q_{0}}$ satisfies the exterior corkscrew condition with constant depending on $N$.
\end{proof}

\subsection{Proof of \eqref{mainb} implies \eqref{maine}}\label{bimpliese}

Suppose \eqref{mainb} holds and let $\Omega_\star:=\Omega_{\F,Q_{0}}$ be a chord-arc domain constructed in Section \ref{bimpliesd} in the proof of \eqref{mainb} implies \eqref{maind}. By \cite[Proposition 6.4, Corollary 3.6]{HM}, we can locate some $A_{Q_0}\in\Omega\cap\Omega_\star$, which is simultaneously a Corkscrew
point for $\Omega_\star$ with respect to $B(x, C\,r_{Q_0})\cap\Omega_\star$ for all $x\in\pom_\star$ and a Corkscrew point for
$\Omega$ with respect to $B(x,C\,r_{Q_0})\cap \Omega$ for all $x\in Q_0$. To prove \eqref{maine}, it suffices to show that there exist constants $\theta$, $\theta'>0$ and $C>1$ (possibly depending on $\Omega_\star$) such that
\begin{equation}\label{e:6-1}
C^{-1}\sigma(F)^{\theta'}
\le
\omega^{A_{Q_0}}(F)
\le
C\,\sigma(F)^{\theta}
\qquad
\forall\,F\subset Q_0\setminus\bigcup_{Q_j\in\F} Q_j.
\end{equation}
By Harnack's inequality, one may replace
$A_{Q_0}$ in \eqref{e:6-1} with some fixed pole $X_{0}$ at the expense of changing the value of $C$. We also remark that technically speaking we should consider arbitrary sets $F\subset \pom\cap\pom_\star$. However, the general case follows from \eqref{e:6-1} in view of \eqref{prop61} and the fact that $\sigma(\partial Q)=\omega^{A_{Q_0}}(\partial Q)=0$  for every $Q\subset\dd_{Q_0}$, since $\sigma$ and $\omega^{A_{Q_0}}$ are doubling (see \cite[Proposition 6.3]{HM} and \cite{HMT}).

Abusing notation, we let $\omega=\omega^{A_{Q_0}}$ denote harmonic measure of $\Omega$ with pole at $A_{Q_0}$ and let
$\omega_{\star}=\omega_{\Omega_\star}^{A_{Q_0}}$ denote harmonic measure of $\Omega_{\star}$ with pole at $A_{Q_0}$. Since $\Omega_\star$ is a chord-arc domain, $\omega_{\star}\in A_{\infty}(\partial\Omega_{\star})$ by \cite{DJ, Sem}. In particular, there exist constants $\alpha$, $\beta$ such that
\begin{align}\label{Ainfty-CAD}
\left(\frac{\sigma_\star(E_\star)}{\sigma (\pom_\star)}\right)^{\alpha}
\lesssim \frac{\omega_\star (E_\star)}{\omega_\star(\pom_\star)}\lesssim \left(\frac{\sigma_\star(E_\star)}{\sigma_\star (\pom_\star)}\right)^{\beta}
\qquad
\forall\, E_\star\subset\pom_\star,
\end{align}
where $\sigma_\star=\rh^n|_{\pom_\star}$,  and we view $\partial\Omega_\star$ as a surface ball in $\partial\Omega_\star$ with arbitrary center in $\partial\Omega_\star$ and radius $\diam(\partial\Omega_\star)$. Recall that $\Omega_\star$ is a chord-arc domain and, in particular, $\sigma_\star (\pom_\star)\approx\diam(\pom_\star)^n\approx \ell(Q_0)^n$.

We need to introduce some notation from \cite[Section 6]{HM}. Given a Borel measure $\mu$ defined on $Q_0$ and $\F$ from above, set
\begin{align}
\label{projopP}
\mc{P}_{\F}\mu(E)
:=
\mu\Big(E\setminus \bigcup\limits_{Q_{j}\in\F} Q_{j}\Big)+\sum\limits_{Q_{j}\in \F} \frac{\sigma(E\cap Q_{j})}{\sigma(Q_{j})}\,\mu(Q_{j})
\qquad \forall\,E\subset Q_{0}.
\end{align}
Also, define a Borel measure on $Q_0$ by the rule
\begin{align}
\label{defnofnu}
\nu(E):
=
\omega_{\star}\Big(E\setminus \bigcup\limits_{Q_{j}\in\F} Q_{j}\Big)+\sum\limits_{Q_{j}\in \F} \frac{\omega(E\cap Q_{j})}{\omega(Q_{j})}\,\omega_{\star}(P_{j})\qquad \forall\,E\subset Q_{0}.
\end{align}
Here $P_j\subset\partial\Omega_\star$ are $n$-dimensional cubes from \cite[Proposition 6.7]{HM}, which satisfy
\begin{align}
\label{prop67}
\ell(P_{j})\approx \dist(P_{j},Q_{j})\approx \dist(P_{j},\partial\Omega)\approx \ell(Q_j)
\qquad\mbox{and}\qquad\sum_j 1_{P_j}\le C
.
\end{align}
From the definition of the projection operator $\mc{P}_{\F}$ in \eqref{projopP} and measure $\nu$ in \eqref{defnofnu},
\begin{align}
\label{projofnu}
\mc{P}_{\F}\nu(E)=\omega_{\star}(E\setminus \bigcup\limits_{Q_{j}\in\F} Q_{j})+\sum\limits_{Q_{j}\in \F} \frac{\sigma(E\cap Q_{j})}{\sigma(Q_{j})}\omega_{\star}(P_{j})\qquad \forall\, E\subset Q_{0},
\end{align}
which depends only on $\omega_{\star}$ and not on $\omega$.

From \cite[Lemma 6.15]{HM}, we have the following version of the Dahlberg-Jerison-Kenig sawtooth lemma (see \cite{DJK}):
There exists some $\theta>0$ such that for every $E\subset Q_0$,
\begin{align}\label{DJK-P}
\left(\frac{\mc{P}_{\F}\omega(E)}{\mc{P}_{\F}\omega(Q_0)}\right)^{\theta}\lesssim \frac{\mc{P}_{\F}\nu(E)}{\mc{P}_{\F}\nu(Q_0)}\lesssim \frac{\mc{P}_{\F}\omega(E)}{\mc{P}_{\F}\omega(Q_0)}.
\end{align}
Note that from \eqref{projopP} and \eqref{projofnu}, for every $F\subset Q_{0}\setminus \cup Q_{j}$  we have $\mc{P}_{\F}\nu(F) =
\omega_{\star}(F)$ and $\mc{P}_{\F}\omega(F)=\omega(F)$.
It is trivial to see that $\mc{P}_{\F}\omega(Q_0)=\omega(Q_0)=\omega^{A_{Q_0}}(Q_0)\approx 1$ by Bourgain's estimate (see \cite{B87} or also \cite{HMT-general}), since $A_{Q_0}$ is an effective Corkscrew point relative to $Q_0$. Also, from \eqref{prop67} and \eqref{projofnu} it is clear that $\mc{P}_{\F}\nu(Q_0)\lesssim \omega_\star(\pom_\star)=1$. Additionally,
\cite[Proposition 6.12, (6.19)]{HM}, the fact that $\omega_\star$ is doubling (see \cite{HMT-general}), and Bourgain's estimate (see \cite{B87} or \cite{HMT-general})  give $$\mc{P}_{\F}\nu(Q_0)\gtrsim \omega_\star(B(x_{Q_0}^\star,C\,\ell(Q_0))\cap\pom_\star)\approx 1,$$ where $x_{Q_0}^\star\in\pom_\star$ and we have used that $A_{Q_0}$ is an effective Corkscrew point relative to $B(x_{Q_0}^\star,C\,\ell(Q_0))\cap\pom_\star$ as noted above. All together, these observations plus \eqref{prop61}, \eqref{Ainfty-CAD} and \eqref{DJK-P} yield
$$
\sigma(F)^\alpha
\lesssim
\omega_\star (F)
\lesssim
\omega(F)
\lesssim
\omega_\star (F)^\frac1{\theta}
\lesssim
\sigma(F)^{\frac{\beta}{\theta}}
\qquad
\forall\,F\subset Q_0\setminus\bigcup_{Q_j\in\F} Q_j,
$$ as desired. We remark that the implicit constants depend on $\sigma (\pom_\star)\approx\ell(Q_0)^n$.

\section{Proof of Theorem \ref{mainvarc}}
\label{variablecase}

As noted in the introduction, to establish Theorem \ref{mainvarc}, it suffices to prove  that \eqref{maincvc}  implies \eqref{maind} and \eqref{mainb} implies \eqref{mainevc}.
Before proceeding to the proof, we pause to make a relevant remark.

\begin{remark}\label{remark:A}
The conditions on the coefficients of the operator in Theorem \ref{mainvarc} are qualitative versions of the corresponding conditions imposed in \cite{KP} and \cite{HMT}. Indeed, for every bounded subdomain $\Omega'$ of $\Omega$, the matrix $A$ satisfies the conditions imposed in \cite{KP} and \cite{HMT} with respect to the domain $\Omega'$. It is worth mentioning that allowing implicit constants to depend on the subdomain $\Omega'$ would be problematic in \cite{KP} and \cite{HMT}, where the authors are interested in establishing scale-invariant estimates. Nevertheless, we allow such dependence below, because our goal is to obtain qualitative rather than quantitative conditions.
\end{remark}

%% proof of c implies d
%%
%%
%%
\subsection{Proof of \eqref{maincvc}  implies \eqref{maind}}\label{ss:cprime-d}
The proof follows the same scheme of Section \ref{cimpliesd} and we only highlight the main changes. We replace $\omega$, $G$ and $\mathcal{G}$ by $\omega_{L}$,
$G_L$, $\mathcal{G}_L$ throughout Section \ref{cimpliesd}. We note that thanks to \cite{HMT-general} we have all the required ``PDE properties'' such as Bourgain and Caffarelli-Fabes-Mortola-Salsa type estimates, doubling of elliptic measure, etc. However, one needs to rework the integration by parts (i.e., the main argument in Lemma \ref{lamma:pack}), since now we work with $L$ in place of the Laplacian. Much as before we have the following substitute of \eqref{sigmaIandII}:
\begin{align*}
(N\,C_1)^{-1} \sigma_L(Q)
&\le
C_1^{-1}\omega_L(Q)
\le \int\limits_{\partial\Omega}\phi_{Q}\,\rd\omega
=
-\iint\limits_{\Omega}A\,\nabla \mc{G}_L\cdot\nabla\phi_{Q}\,\rd X\\
&=
-\iint\limits_{\Omega}(A\,\nabla \mc{G}_L-\vec{\alpha})\cdot\nabla \phi_{Q}\,\rd X
+
\iint\limits_{\Omega_{\rm ext}}\vec{\alpha}\cdot\nabla \phi_{Q}\,\rd X
\\
&=:
-\mathcal{I}+\mathcal{II},
\end{align*}
where this time $\vec{\alpha}:=\frac{1}{|U_{Q,\epsilon }|}\iint\limits_{U_{Q,\epsilon }} A\,\nabla \mc{G}_L\rd X$. From here the proof continues \textit{mutatis mutandis} (again with the help of \cite{HMT-general} for the required ``PDE properties'') and with the harmless presence of the matrix $A$ up to \eqref{trfrg}, which eventually leads to the following variable coefficient version of \eqref{sigmaQGreen}:
\begin{align*}
\sigma(Q) & \lesssim_{N} \iint\limits_{U_{Q, \epsilon }} |\nabla(A\,\nabla \mc{G}_L)|^{2}\,\mc{G}_L\, \rd X
\qquad
\forall\,
Q\in \dd_{\F_N, Q_0}\cap \mc{B}^N.
\end{align*}
Taking this into account and following the same argument, what is left to prove is
\begin{equation} \label{e:3-1}
\iint\limits_{\Omega^{\star}_M} |\nabla(A\,\nabla \mc{G}_L)|^{2}\,\mc{G}_L\rd X
\lesssim_N
\sigma(Q_0')
\end{equation}
with constants that do not depend on $M$. It was in that part of the proof where we strongly used harmonicity. However, \eqref{e:3-1} follows  from the following ``integration by parts'' proposition in \cite{HMT}.

\begin{proposition}[\cite{HMT}]\label{prop:CME-G}
Under the current background hypotheses and notation,  and
with Remark \ref{remark:A} in view, given $\Theta\ge 1$,
and a family $\F_K\subset \dd_{Q_0'}$, $K\ge 1$,
of pairwise disjoint dyadic cubes satisfying
\begin{equation}\label{goal:N:relevant}
\Theta^{-1}
\le
\frac{\omega_L(Q)}{\sigma(Q)} \le \Theta
\qquad\mbox{and}\qquad
\ell(Q)>2^{-K}\,\ell(Q_0')
\qquad \forall\,Q\in \dd_{\F_{K},Q_0'},
\end{equation}
we then have
\begin{equation}\label{eqn-goal:N}
\frac1{\sigma(Q_0')}\iint_{\Omega_{\F_{K},Q_0'}^*} |\nabla
(A\,\nabla \mathcal{G}_L)(X)|^2\,\mathcal{G}_L(X)\,dX
\le
C \,,
\end{equation}
where $C$ depends on $\Theta$, $Q_0$ and the allowable parameters, but not on $K$.
\end{proposition}

Note that the constants in the analogue of \eqref{frefr} in the variable coefficient setting may depend on $N$ and so does depend the implicit constant in \eqref{e:3-1}. This completes the proof of the variable coefficient version of Lemma \ref{lamma:pack} . The rest of the argument is of a geometrical nature and can be carried out without change. Details are left to the interested reader.

For the sake of completeness we sketch the integration by parts argument from \cite{HMT}. We first notice  that as observed before the ADR assumption implies the ``capacity density condition'' and therefore the required PDE properties that we will be using (Bourgain, Caffarelli-Fabes-Mortola-Salsa, etc.) hold. The arguments are somehow similar to the ones in \cite{JK} and the full details will be given in \cite{HMT-general}. Another relevant observation is that our current assumptions (i.e, the Carleson estimate \eqref{main-A-Car}) imply that for any positive weak  solution $u$ on $2\,B\cap \Omega$, where $B$ is a ball centered at $\pom$, one has the pointwise estimate $|\nabla u(X)|\le C_B\,u(X)/\delta(X)$. This follows from  re-scaling \cite[Lemma 3.1]{GW} (see \cite{HMT} for full details). This and the hypotheses in Proposition \ref{prop:CME-G} guarantee that the following holds
\begin{equation}
|\nabla  \mathcal{G}_L(X)|
\lesssim_{Q_0}
\frac{ \mathcal{G}_L(X)}{\delta(X)}
\approx\frac{\omega_L(Q)}{\sigma(Q)} \lesssim_\Theta 1,
\qquad
\forall\, X\in \overline{\Omega_{\F_{K},Q_0'}^*} .
\label{eq:est-grad-G-saw}
\end{equation}
On the other hand, by \eqref{main-A-Car} and elementary geometric arguments one can see that 
$\big\||\nabla A(\cdot)|\,\delta(\cdot)\big\|_{L^\infty(B_{Q_0^*}\cap\Omega)}\lesssim C_{2\,B_{Q_0^*}}$ where we recall that $B_{Q_0}^*=\kappa_0\, B_{Q_0}$. 

We can now estimate \eqref{eqn-goal:N}. First, when $\nabla$ hits $A$ one can see that
\begin{multline*}
\iint_{\Omega_{\F_{K},Q_0'}^*} 
|\nabla A(X)|^2\, |\nabla \mathcal{G}_L(X)|^2\,\mathcal{G}_L(X)\,dX
\lesssim
\iint_{\Omega_{\F_{K},Q_0'}^*} 
|\nabla A(X)|^2\,\delta(X)\,dX
\\
\lesssim
\big\||\nabla A(\cdot)|\,\delta(\cdot)\big\|_{L^\infty (B_{Q_0^*}\cap \Omega)}\,
\iint_{T_{Q_0'}} |\nabla A(X)|\,dX
\lesssim
C_{2\,B_{Q_0^*}}^2\,\ell(Q_0')^{n},
\end{multline*}
where in the last estimate we have used \eqref{main-A-Car} along with the facts that $Q_0'\subset Q_0$ and 
$T_{Q_0'}\subset B_{Q_0'}^*\cap\Omega$. To continue with \eqref{eqn-goal:N} we now look at the terms that appear when $\nabla$ hits $\nabla\mathcal{G}_L$. In that case we can simply use that $A$ is a bounded matrix and hence it suffices to show that
\begin{equation}
\iint_{\Omega_{\F_{K},Q_0'}^*} |\nabla (\partial \mathcal{G}_L)(X)|^2\,\mathcal{G}_L(X)\,dX
\le
C\, \sigma(Q_0),
\label{eq:calim-IBP-A}
\end{equation}
where ``$\partial$'' denotes a fixed generic derivative and 
where $C$ depends on $\Theta$, $Q_0$ and the allowable parameters, but not on $K$. Next we use that $A$ is uniformly elliptic, symmetric and that $\mathcal{G}_L$ is a weak solution in $\overline{\Omega_{\F_{K},Q_0'}^*}$ (let us recall that we always work in a regime where the fixed pole $X_0\notin 4\,B_{Q_0^*}$) to obtain that
\begin{multline}\label{form-ibp}
|\nabla (\partial \mathcal{G}_L)|^2\,\mathcal{G}_L
\lesssim
A\,\nabla (\partial \mathcal{G}_L)\cdot \nabla (\partial \mathcal{G}_L)\,\mathcal{G}_L
%\\
%=
%\partial(A\,\nabla \mathcal{G}_L)\cdot \nabla(\mathcal{G}_L\,\partial\mathcal{G}_L)
%-
%\partial A\,\nabla \mathcal{G}_L\,\nabla(\mathcal{G}_L\,\partial\mathcal{G}_L)
%-
%\frac12
%\div\big(A\,\nabla \mathcal{G}_L\, (\partial \mathcal{G}_L)^2\big)
\\
=
\partial\big(A\,\nabla \mathcal{G}_L \cdot \nabla(\mathcal{G}_L\,\partial\mathcal{G}_L)
\big)
-
\div\big(A\,\nabla \mathcal{G}_L\, \partial(\mathcal{G}_L\,\partial \mathcal{G}_L)\big)
\\
-
\partial A\,\nabla \mathcal{G}_L\,\nabla(\mathcal{G}_L\,\partial\mathcal{G}_L)
-
\frac12
\div\big(A\,\nabla \mathcal{G}_L\, (\partial \mathcal{G}_L)^2\big)
\end{multline}
Notice that we can use the divergence theorem in the domain ${\Omega_{\F_{K},Q_0'}^*}$ since the latter 
is a finite union of fattened Whitney cubes, so its boundary consists of portions of faces of those cubes and thus 
its (outward) unit normal $\nu$ is well defined a.e.~on $\partial\Omega_{\F_{K},Q_0'}^*$. This, \eqref{form-ibp}, and the fact that $A$ is bounded allow us to conclude that
\begin{align*}
I:&=\iint_{\Omega_{\F_{K},Q_0'}^*} |\nabla (\partial \mathcal{G}_L)(X)|^2\,\mathcal{G}_L(X)\,dX
\\
&\lesssim
\int_{\partial \Omega_{\F_{K},Q_0'}^*}|\nabla\mathcal{G}_L|^3\,d\mathcal{H}^n
+
\int_{\partial \Omega_{\F_{K},Q_0'}^*} |\nabla \mathcal{G}_L|\,|\nabla ^2 \mathcal{G}_L|\,\mathcal{G}_L\,d\mathcal{H}^n
\\
&\qquad\qquad+
\iint_{\Omega_{\F_{K},Q_0'}^*} |\nabla A|\,|\nabla\mathcal{G}_L|^3\,dX
+
\iint_{\Omega_{\F_{K},Q_0'}^*} |\nabla A|\,|\nabla\mathcal{G}_L|\,\nabla(\partial\mathcal{G}_L)|\,\mathcal{G}_L\,dX
\\
&
=:II+III+IV+V.
\end{align*}
To estimate $II$ we invoke \eqref{eq:est-grad-G-saw} and use that $\partial\Omega_{\F_{K},Q_0'}^*$ is ADR (since it is a sawtooth domain) with bounds that depend only on the ADR constant of $\pom$ and the other allowable parameters of $\Omega$ (but that are independent of $K$), see \cite[Lema 3.61]{HM}:
$$
II
\lesssim_{Q_0, \Theta}
\mathcal{H}^n(\partial\Omega_{\F_{K},Q_0'}^*)
\approx
\diam(\partial\Omega_{\F_{K},Q_0'}^*)^n
\approx
\ell(Q_0')^n
\approx
\sigma(Q_0').
$$
For $IV$ we use again \eqref{eq:est-grad-G-saw} and \eqref{main-A-Car}:
$$
IV
\lesssim_{Q_0, \Theta}
\iint_{B_{Q_0'}^*\cap \Omega} |\nabla A|\,dX
\lesssim
C_{B_{Q_0}^*}\,\ell(Q_0')^n
\approx
C_{B_{Q_0}^*}\,\sigma(Q_0').
$$
For $V$ we apply Young's inequality and once more \eqref{eq:est-grad-G-saw} 
\begin{align*}
V
&\le
\frac12
\iint_{\Omega_{\F_{K},Q_0'}^*} |\nabla(\partial\mathcal{G}_L)|^2\,\mathcal{G}_L\,dX
+\frac12
\iint_{\Omega_{\F_{K},Q_0'}^*} |\nabla A|^2\,|\nabla\mathcal{G}_L|^2\,\mathcal{G}_L\,dX
\\
&\le
\frac12\,I
+
C_{Q_0, \Theta}\,\big\||\nabla A(\cdot)|\,\delta(\cdot)\big\|_{L^\infty(B_{Q_0^*}\cap \Omega)}
\iint_{B_{Q_0'}^*\cap \Omega} |\nabla A|\,dX
\\
&\le
\frac12\,I
+
C_{Q_0, \Theta}\,C_{2\,B_{Q_0}^*}\,C_{B_{Q_0}^*}\,\ell(Q_0')^n
\\
&\le
\frac12\,I
+
C_{Q_0, \Theta}'\,\sigma(Q_0').
\end{align*}
At this stage to complete our proof we just need to see that $III$ satisfies an estimate like $II$. In such a case we collect all the previous computations, absorb $\frac12 I$ from $V$ and \eqref{eq:calim-IBP-A} follows. 

Let us now explain how to estimate $III$, a more delicate term because of the presence of $|\nabla^2 \mathcal{G}_L|$. Note that in the harmonic case, the second derivative of a harmonic function satisfies interior estimates since the first derivative is harmonic. Here $\mathcal{G}_L$ is a solution and hence we can control $\nabla \mathcal{G}_L$ as in \eqref{eq:est-grad-G-saw}, but a pointwise estimate for $\nabla^2 \mathcal{G}_L$ is not expected. However, we would like to convince the reader that $III$ behaves essentially as $II$, the full rigorous argument will appear in the forthcoming paper \cite{HMT}, here we just sketch the main ideas.  First, it is not difficult to show that one can obtain a Caccioppoli type estimate for $|\nabla^2 \mathcal{G}_L|$ in terms of $|\nabla\mathcal{G}_L|$ by using the $L^\infty$ bounds for $|\nabla A(\cdot)|\,\delta(\cdot)$ observed above. Hence to estimate  $III$ we would need to replace the integral along the boundary of the sawtooth by some kind of solid integration running around the boundary. This can be done if we repeat the ``integration by parts argument'' given above 
by incorporating a partition of unity adapted to the Whitney regions $U_Q$ that form the sawtooth domain that we are considering. All boundary terms become now solid integrals over some collection of Whitney cubes covering the boundary of the sawtooth. In that case we can estimate $|\nabla^2 \mathcal{G}_L|$ and see that $III$ obeys the same bound as $II$. The reader can find a similar ``integration by parts argument'' with a partition of unity at \cite[Section 5.2]{HLMN}.

\subsection{Proof of \eqref{mainb} implies \eqref{mainevc}} \label{ss:b-cprime}
We follow the arguments given in Section \ref{bimpliese} above and indicate the necessary changes. To that end, set $\omega_{L}=\omega^{A_{Q_{0}}}_L$ and $\omega_{L,\star}=\omega^{A_{Q_{0}}}_{L,\Omega_{\star}}$, which will replace $\omega=\omega^{A_{Q_{0}}}$ and $\omega_{\star}=\omega^{A_{Q_{0}}}_{\Omega_{\star}}$, respectively.
%The proof requires some changes that we explain next.

We claim that $\omega_{L,\star}\in A_\infty(\pom_\star)$. This is nowadays folklore and follows from the fact that $\Omega_\star$ is a bounded chord-arc domain
and from the properties of the matrix $A$. Let us sketch the argument for the sake of completeness. First recall that Kenig and Pipher showed in \cite{KP} that if $\widehat{\Omega}$ is a bounded Lipschitz domain, then its associated elliptic measure $\omega_{L,\widehat{\Omega}}\in A_\infty(\partial \widehat{\Omega})$ provided that
$$
a_{\widehat{\Omega}}(X)
=
\sup_{Y\in B(X,\delta_{\widehat{\Omega}}(X)/2)} |\nabla A(Y)|^2\,\delta_{\widehat{\Omega}}
(Y)
$$
is a Carleson measure in $\widehat{\Omega}$. Here $\delta_{\widehat{\Omega}}$ denotes the distance to $\partial\widehat{\Omega}$. It is straightforward to show that our assumptions on $A$  give such a condition for every bounded subdomain $\widehat{\Omega}\subset\Omega$, see Remark \ref{remark:A}. On the other hand, since $\Omega_\star$ is a bounded chord-arc domain, $\Omega_\star$ satisfies an
``interior big pieces'' of Lipschitz sub-domains condition by the work of David-Jerison in \cite{DJ}. Together with a simple maximum principle argument (see e.g.~ in \cite{Dah,DJ}) and the aforementioned result of \cite{KP},
one quickly obtains the following lower bound for $\omega_{L, \Omega_\star}$:
There are constants $\eta\in (0,1)$ and $c_0>0$ such that for
each surface ball $\Delta_\star\subset\partial\Omega_\star$, and any Borel subset $F\subset\Delta_\star$,
\begin{equation}\label{eq1.2***}
\omega_{L, \Omega_\star}^{X_{\Delta_\star}}(F)\geq c_0\,,\qquad {\rm whenever}\,\,\,\sigma_\star(F)\geq \eta \,\sigma_\star(\Delta_\star).
\end{equation}
In turn, the latter bound self-improves to an $A_\infty$ estimate for $\omega_{L, \Omega_\star}$ as desired,
via the comparison principle (see, e.g., \cite{DJ}).

Once this claim has been obtained, the proof follows \textit{mutatis mutandis} with a version of \eqref{DJK-P} for $L$ and with some needed PDE tools that can be taken from \cite{HMT-general}. Details are left to the interested reader.

%%
%%
%% proof of b impies e done!

\section{An example}\label{section:example}

As noted in the introduction, Theorem \ref{main} can be seen as a qualitative version of \eqref{AHMNT} in the sense that \eqref{maina}--\eqref{maine} can be seen as qualitative versions of (i)--(iv) in \eqref{AHMNT}. In view of this, it is worthwhile to find an example of a domain $\Omega_\star$ satisfying the required background hypotheses (i.e., 1-sided NTA with ADR boundary), for which \eqref{maina}--\eqref{maine} in Theorem \ref{main} hold, but (i)--(iv) in \eqref{AHMNT} fail. In particular, the corresponding harmonic measure (or elliptic measures of the previous section) will satisfy the absolute continuity conditions \eqref{mainc} and \eqref{maine}, but will not belong to $A_\infty$ or weak-$A_\infty$. To find this example, we will start with $\Omega\subset\re^3$, a 1-sided NTA domain with empty exterior, whose boundary is not rectifiable. We will then define a sawtooth subdomain $\Omega_\star$ whose boundary is a countable union of partial faces of Whitney boxes. Therefore, $\partial\Omega_\star$ is rectifiable. On the other hand, we will see that $\partial\Omega_\star$ cannot be Uniformly Rectifiable, otherwise $\Omega_\star$ is NTA and satisfies the exterior corkscrew condition by \eqref{AHMNT} (in particular, see \cite{AHMNT}), but our construction  prevents this from happening.

Let $\mathfrak{C}$ be the ``4-corner Cantor set'' of J. Garnett (see,  e.g., \cite[p. 4]{DS2}). It is not difficult to show  from the construction of the set that $\mathbb{R}^{2}\setminus\mathfrak{C}$ is a 1-sided NTA domain with 1-dimensional ADR boundary and empty exterior.
Let $\mathfrak{C}^{\star}=\mathfrak{C}\times\mathbb{R}$ and $\Omega=\mathbb{R}^{3}\setminus \mathfrak{C}^{\star}$. We shall show that $\Omega$ is a 1-sided NTA with 2-dimensional ADR boundary.

Let us first show that $\partial\Omega=\mathfrak{C}^\star$ is 2-dimensional ADR. Given $x'\in\re^2$ and $r>0$ we write $B_2(x',r)\subset \re^2$ to denote the 2-dimensional ball centered at $x'$ with radius $r$. Analogously, given $t\in\re$ and $r>0$, $B_1(t,r)\subset\re$ denotes the 1-dimensional interval centered at $t$ with radius $r$. It is clear from the definition that for every $x=(x',t)\in\pom=\mathfrak{C}^\star$ one has
$$
\big(B_2\big(x',r/\sqrt{2}\,\big)\cap\mathfrak{C}\big)\times B_1\big(t,r/\sqrt{2}\,\big)\subset B(x,r)\cap\mathfrak{C}^\star\subset \big(B_2(x',r)\cap\mathfrak{C}\big)\times B_1(t,r).
$$
These readily imply that $\mathfrak{C}^\star$ is 2-dimensional ADR as $\mathfrak{C}$ is 1-dimensional ADR. We now recall that the complement of an ADR set always satisfies the Corkscrew condition. In particular, so does $\re^3\setminus\mathfrak{C}^\star=\Omega$ and this gives the (interior) Corkscrew condition for $\Omega$. We next show that $\Omega$  satisfies  the Harnack Chain condition (in $\re^3$), again using that $\re^2\setminus\mathfrak{C}$ has the same property in $\re^2$.  To this end let $\rho>0$, $\Lambda\geq 1$ be given.  Let $X=(X',t)$, $Y=(Y',s)\in\Omega$, with $X'$, $Y'\in\mathbb{R}^{2}$ and $t$, $s\in\mathbb{R}$. Note that in particular $X'$, $Y'\in\re^2\setminus\mathfrak{C}$. Assume that $\delta(X),\delta(Y)>\rho$ and $|X-Y|\le \Lambda \rho$. It can be easily seen that
\begin{align}\label{distancetoCantorset}
\delta(X)=\dist(X, \partial\Omega)=\dist(X',\mathfrak{C})\quad \mbox{and}\quad \delta(Y)=\dist(Y,\partial\Omega)=\dist(Y',\mathfrak{C}).
\end{align}
Using this observation  and that $\mathbb{R}^{2}\setminus\mathfrak{C}$ satisfies the Harnack Chain condition in $\re^2$,
we can find a Harnack Chain of 2-dimensional open balls connecting $X'$ and $Y'$. The implicit constants are all under control and in particular the number of balls depends only on $\Lambda$ by \eqref{distancetoCantorset} and since $|X'-Y'|\le |X-Y|\le \Lambda\,\rho$. Now, for any of the balls $B^i_2=B_2(X_i',r_i)\subset\re^2$, we let $B_i:=B(X_i,r_i)\subset\re^3$ with $X_i=(X_i',t)$. Clearly $\{B_i\}_i$ is a Harnack chain in $\Omega$ connecting $X=(X',t)$ and $Z:=(Y',t)$. Next we can add to the previous chain a collection of 3-dimensional balls, connecting $Z=(Y',t)$ with $Y=(Y',s)$, whose centers lie in the line segment between $Z$ and $Y$ and with radius equal to $\delta(Y)/2=\dist(Y',\mathfrak{C})/2$. Note that the number of balls will be of the order of $|t-s|/\delta(Y)\le |X-Y|/\rho\le \Lambda$. Hence our proof of the Harnack Chain condition is complete.

Now that we know $\Omega$ is a 1-sided NTA domain with ADR boundary, we simply note that $\Omega_{\rm ext}= \emptyset$, which in particular implies  \eqref{mainb} in Theorem \ref{main} fails. Therefore, $\sigma$ is not absolutely continuous with respect to $\omega$  by Theorem \ref{main}; that is, there is a set $F\subset\mathfrak{C}^\star=\pom$ with $\sigma(F)>0$ but $\omega(F)=0$. Also, $\mathfrak{C}^\star=\pom$ is not rectifiable,   again by Theorem \ref{main}.

Next,  we are going to construct a pairwise disjoint family $\mathcal{F}\subset \dd_{Q_0}$ with $Q_0$ a given dyadic cube in $\pom=\mathfrak{C}^{\star}$ and let $\Omega_{\star}=\Omega_{\mc{F},Q_0}$. As described above, we shall show that
\begin{align} \label{omegaisntaadrr}
\begin{split}
&\Omega_{\star}\, \,  \mbox{is 1-sided NTA}, \\
&\partial\Omega_{\star}\, \, \mbox{is ADR},\\
&\partial\Omega_{\star}\, \, \mbox{is rectifiable},\\
&\mbox{$\partial\Omega_{\star}$ is not Uniformly Rectifiable}.
\end{split}
\end{align}
Notice that in such a case $\omega_{\Omega_\star}$ will not be weak-$A_\infty(\pom_\star)$ by \eqref{AHMNT} (in particular, see \cite{HMU}), but  \eqref{maina}--\eqref{maine} in Theorem \ref{main} hold, so in particular $\mathcal{H}^n\big|_{\pom_\star}\ll \omega_{\Omega_\star}$ by \eqref{mainc} and the two measures are mutually absolutely continuous $\mathcal{H}^n\big|_{\pom_\star}$-a.e.~by \eqref{maine}.

To construct our example, we let $\alpha_0\in\N$ large enough  and let $\{\alpha_{k}\}_{k\ge 1}\subset \N$ be an strictly increasing sequence such that $\alpha_k\to \infty$ fast enough as $k\to\infty$.  Fix $Q_0\in\dd$ such that $\ell(Q_0)=2^{-\alpha_0}$ and set $x_0=x_{Q_0}$ its center.  Take $Q_1\in\dd_{Q_0}$ such that $x_0\in Q_1$ and $\ell(Q_1)=2^{-\alpha_1}$. Write $x_1=x_{Q_1}$ and set
$$
\F_1
:=
\{Q\in\dd_{Q_0}:\ \ell(Q)=2^{-\alpha_{1}},\ x_0\notin Q\}
=
\{Q\in\dd_{Q_0}:\ \ell(Q)=2^{-\alpha_{1}}\}\setminus\{Q_1\}
,
$$
which is a pairwise disjoint family. Next, take $Q_2\in\dd_{Q_1}$ such that $x_1\in Q_2$ and $\ell(Q_2)=2^{-\alpha_2}$. Let $x_2=x_{Q_2}$ denote the center of $Q_2$ and set
$$
\F_2
:=
\{Q\in\dd_{Q_1}:\ \ell(Q)=2^{-\alpha_{2}},\ x_1\notin Q\}
=
\{Q\in\dd_{Q_1}:\ \ell(Q)=2^{-\alpha_{2}}\}\setminus\{Q_2\}
,
$$
which is again a pairwise disjoint family. Iterating this we have a family of cubes $\{Q_k\}_{k\ge 0}\subset \dd_{Q_0}$ such that $\ell(Q_k)=2^{-\alpha_k}$, and $Q_0\supsetneq Q_1\supsetneq Q_2\dots$ and $x_k=x_{Q_k}\in Q_{k+1}$. In particular, $\{x_k\}_{k\ge 1}$ is a Cauchy sequence (since $\alpha_k\to\infty$ as $k\to\infty$), and hence, there exists $\bar{x}\in\pom$ such that $x_k\to x$ as $k\to\infty$. Note that $\bigcap_{k\ge 0}\overline{Q_k}=\{\bar{x}\}$. Our construction also gives a family of pairwise disjoint families $\{\F_k\}_{k\ge 1}$ such that  $\F:=\bigcup_{k\ge 1}\F_k$ is a pairwise disjoint family of dyadic subcubes of $Q_0$, and
$$
\bigcup_{Q\in\F} Q
=
\bigcup_{k=1}^\infty\bigcup_{Q\in\F_k} Q
=
\bigcup_{k=1}^\infty (Q_{k-1}\setminus Q_k)
=
Q_0\setminus\bigcap_{k=1}^\infty Q_k.
$$
At last,  we set $\Omega_\star=\Omega_{\F,Q_0}$ (see Section \ref{s1}). Because $\Omega$ is a 1-sided NTA domain with ADR boundary, we know that $\Omega_\star$ is also a 1-sided NTA domain with ADR boundary by \cite[Lemma 3.61]{HM}. It remains to show that $\partial\Omega_\star$ is rectifiable, but not Uniformly Rectifiable.

To see that $\partial\Omega_\star$ is rectifiable, note that
\begin{align*} \label{decomposeOstar}
\partial\Omega_\star=(\partial\Omega\cap\partial\Omega_\star)\cup (\Omega\cap \partial\Omega_\star).
\end{align*}
On the one hand, from  \eqref{prop61}, \cite[Proposition 6.3]{HM}, and the fact that $\pom$ is ADR (and hence $\sigma=\mathcal{H}^2\big|_{\pom}$ is doubling), it follows that
$$
\sigma(\partial\Omega\cap\partial\Omega_\star)
=
\sigma\Big(Q_0\setminus\bigcup_{Q\in\F} Q\Big)
=
\sigma\Big(\bigcap_{k=1}^\infty Q_k\Big)
=
\lim_{k\to\infty} \sigma(Q_k)
=0,
$$
where the last equality  holds since $\sigma(Q_k)\approx\ell(Q_k)^2=2^{-2\,\alpha_k}\to 0$ as $k\to\infty$. On the other hand, $\Omega\cap\partial\Omega_\star$ is countable union of partial faces of fattened Whitney cubes and hence $\Omega\cap\partial\Omega_\star$ is rectifiable (see the definition of sawtooth regions in Section \ref{ss:grid}). Therefore,  $\partial\Omega_\star$ is rectifiable.

Finally, we show that $\partial\Omega_\star$ is not Uniformly Rectifiable. Suppose otherwise that $\partial\Omega_\star$ is Uniformly Rectifiable. Then, since $\Omega_\star$ is also a 1-sided NTA domain with ADR boundary, we can apply \eqref{AHMNT} (in particular, the result of \cite{AHMNT}) to conclude that $\Omega_{\star}$ is an NTA domain. Therefore, $\Omega_\star$ satisfies the exterior corkscrew condition  with some constant $c_0$. In particular, for every $x\in\pom_\star$ and $0<r<\diam(\pom_\star)\approx\ell(Q_0)=2^{-\alpha_0}$, there exists
$X_{\Delta_\star}$ such that $B(X_{\Delta_\star},c_0\,r)\subset B(x,r)\cap(\Omega_\star)_{\rm ext}$, where $\Delta_\star=B(x,r)\cap\pom_\star$.  Thus,
\begin{equation}
\label{eq:viola}
\frac{|B(x,r)\cap(\Omega_\star)_{\rm ext}|}{|B(x,r)|}
\ge
\frac{|B(X_{\Delta_\star},c_0\,r)|}{|B(x,r)|}
=
c_0^n
\qquad\forall\,x\in\pom_\star,\quad \forall\, 0<r\lesssim 2^{-\alpha_0}.
\end{equation}
We are going to see that this is violated.  Recall that $x_k\to \bar{x}$ as $k\to\infty$ and that $x_k\in \pom\subset\re^3\setminus \Omega_\star$. Also, for every $k$, we have that $|X_{Q_k}-x_k|\approx\ell(Q_k)=2^{-\alpha_k}$ and hence $X_{Q_k}\to\bar{x}$ as $k\to\infty$. By construction, $X_{Q_k}\in \interior(U_{Q_k})\subset\Omega_\star$, since $Q_k\in\dd_{\F,Q_0}$. All together, these show that $\bar{x}\in\pom_\star$.

To get a contradiction we set $B_k=B(\bar{x},2^{-\alpha_k}/N)$ with $N\ge 1$ large enough to be chosen (depending only on dimension and the 1-sided NTA and ADR constants of $\Omega$). Our goal is to obtain that \eqref{eq:viola} cannot hold, see Figure \ref{corkscrewomegafigure}. Take $X\in (B_k\cap\Omega)\setminus\Omega_\star$. In particular, $X\in\Omega$ and we can take $I\in\W$ such that $I\ni X$ and $\ell(I)\approx\delta(X)\le 2^{-\alpha_k}/N$. Let $Q_I\in\dd$ be the nearest cube to $I$ with $\ell(Q_I)=\ell(I)$ so that $I\in \W_{Q_I}^*$. Then, using the notation in \eqref{DeltaQ},
\begin{multline*}
|x_{Q_I}-x_{k}|
\lesssim
\ell(Q_I)+\dist(Q_I,I)+\ell(I)+|X-\bar{x}|+ |\bar{x}-x_{k+1}|+|x_{k+1}-x_k|
\\
\lesssim
\frac{2^{-\alpha_k}}{N}
+\diam(Q_{k+1})
\approx
\frac{2^{-\alpha_k}}{N}+2^{-\alpha_{k+1}}
<
r_{Q_k}
\end{multline*}
by taking $N$ large enough, because we have assumed that $\alpha_k\to\infty$ fast enough. This implies that $x_{Q_I}\in\Delta_{Q_k}\subset Q_k$, see  \eqref{DeltaQ}. Also, $\ell(Q_I)=\ell(I)\lesssim 2^{-\alpha_k}/N<\ell(Q_k)$, and by the dyadic properties, we conclude that $Q_I\in\dd_{Q_k}$. In particular, observe that $Q_I\in\dd_{Q_0}$. Since $I\in \W_{Q_I}^*$ and $X\notin\Omega_\star=\Omega_{\F,Q_0}$, it follows that $Q_I\notin\dd_{\F,Q_0}$. In other words, $Q_I\in\dd_{\widetilde{Q}}$ for some $\widetilde{Q}\in\F=\bigcup_{k\ge 1}\F_k$. Hence $\widetilde{Q}\in\F_j$ for some $j\ge 1$. By construction, $Q_I\subset Q_{j-1}\setminus Q_j$. Thus, since $Q_I\in\dd_{Q_k}$, it follows that $j\ge k+1$
and $\delta(X)\approx\ell(I)=\ell(Q_I)\le\ell(\widetilde{Q})\le 2^{-\alpha_{k+1}}$. We have shown that
$$
(B_k\cap\Omega)\setminus\Omega_\star
\subset
B_k\cap\{X\notin\pom: \delta(X)\lesssim 2^{-\alpha_{k+1}}\}.
$$
Using \eqref{eq:viola} and \cite[Lemma 5.3]{HM}  (applied to $E=\pom$, which is 2-dimensional ADR), we obtain a contradiction:
$$
c_0^n
\le
\frac{|B_k\setminus\Omega_\star|}{|B_k|}
=
\frac{|(B_k\cap\Omega)\setminus\Omega_\star|}{|B_k|}
\lesssim
\frac{2^{-\alpha_{k+1}}\,(2^{-\alpha_k})^2}{(2^{-\alpha_k})^3}
=
2^{-(\alpha_{k+1}-\alpha_k)}\to 0
$$
as $k\to\infty$ by choosing $\alpha_k\to \infty$ fast enough. Notice that in the first equality we have used that $\re^3=\Omega\cup\pom$ (since $\Omega_{\rm ext}$ is the null set) and that $|\pom|=0$ since $\pom$ is ADR. We have reached a contradiction, and consequently, $\pom_{\star}$ is not Uniformly Rectifiable. This completes the proof of all the items in \eqref{omegaisntaadrr}.

\def\firstcircle{(0,0) circle (2.7cm)}
\def\secondcircle{(0,0) circle (3cm)}
\begin{figure}[ht]
\centering
\begin{tikzpicture}[scale=0.6]

\tikzstyle{arrow} = [thick,->,>=stealth]

\draw (-10,0) rectangle (-4,6) node[midway, opacity=.5] {$Q\in \mathcal{F}_{k}$};
\draw[red, solid, line width=.5mm] (-10,6)--(-4,6);

\draw (10,0) rectangle (4,6) node[midway, opacity=.5] {$Q\in \mathcal{F}_{k}$};
\draw[red, solid, line width=.5mm] (10,6)--(4,6);
\begin{scope}
\draw[gray] (4,0) rectangle (3,1) ;
\draw[red, solid, line width=.5mm] (-4,6)--(-4,1);

\draw[gray] (3,0) rectangle (2,1);

\draw[gray] (2,0) rectangle (1,1);

\draw[gray] (-4,0) rectangle (-3,1);
\draw[red, solid, line width=.5mm] (4,6)--(4,1);

\draw[gray] (-3,0) rectangle (-2,1);
\draw[gray] (-2,0) rectangle (-1,1);

\draw[red, solid, line width=.5mm] (4,1)--(1,1);

\draw[red, solid, line width=.5mm] (-4,1)--(-1,1);

\draw[red, solid, line width=.5mm] (-1,1)--(-1,.25);

\draw[red, solid, line width=.5mm] (1,1)--(1,.25);

\draw[red, solid, line width=.5mm] (-1,.25)--(-.25,.25);

\draw[red, solid, line width=.5mm] (1,.25)--(.25,.25);
\end{scope}

\begin{scope}[scale=0.5, shift={(-1,0)}]
\draw (3,0) rectangle (2.5,0.5);
\draw (2.5,0) rectangle (2,0.5);
\draw (2,0) rectangle (1.5,0.5);
\end{scope}
\begin{scope}[scale=0.5, shift={(1,0)}]

\draw (-3,0) rectangle (-2.5,0.5);
\draw (-2.5,0) rectangle (-2,0.5);
\draw (-2,0) rectangle (-1.5,0.5);
\end{scope}

\node[below] at (0,0) {$\bar{x}$};
\filldraw (0,0) circle (.5pt);

\node[below,purple] at (0,-3) {$B(\bar{x}, 2^{-\alpha_{k}}/N)=B_{k}$};
\draw[red] \firstcircle;
\node at (3,4) {$\mbox{\Large $\Omega_{\star}$}$};
%%%%

\draw[arrow, gray, dashed, opacity=.5] (-3.5, 0.5)--(-3.5,-0.5);
\node[opacity=.5,below] at (-3.5,-0.5) {$Q\in \mathcal{F}_{k+1}$};
\draw[arrow, gray, dashed, opacity=.5] (3.5, 0.5)--(3.5,-0.5);
\node[opacity=.5,below] at (3.5,-0.5) {$Q\in \mathcal{F}_{k+1}$};
\draw[dashed, opacity=.4] (-4,2)--(-4,5);
\draw[dashed, opacity=.4] (4,2)--(4,5);
\draw [opacity=.4, decorate,decoration={brace,amplitude=10pt}]
(-4,6.1) -- (4,6.1) node [black,midway,yshift=0.6cm]
{$\scriptstyle 2^{-\alpha_{k}}$};
\draw [opacity=.4, decorate,decoration={brace,amplitude=10pt}]
(-4,0) -- (-4,6) node [black,midway,xshift=-0.6cm]
{$\scriptstyle 2^{-\alpha_{k}}$};
\draw[dashed, opacity=.4] (-1,1)--(-1,5.3);
\draw[dashed, opacity=.4] (1,1)--(1,5.3);
\draw [opacity=.4, decorate,decoration={brace,amplitude=6pt}]
(-1,5.4) -- (1,5.4) node [black,midway,yshift=0.4cm]
{$\scriptstyle 2^{-\alpha_{k+1}}$};
\draw [opacity=.4, decorate,decoration={brace,amplitude=4pt,mirror}]
(4,0) -- (4,1) node [black,midway,xshift=0.6cm]
{$\scriptstyle 2^{-\alpha_{k+1}}$};
\draw[dashed, opacity=.4] (-.25,.25)--(-.25,4.5);
\draw[dashed, opacity=.4] (.25,.25)--(.25,4.5);
\draw [opacity=.4, decorate,decoration={brace,amplitude=2pt}]
(-.25,4.6) -- (.25,4.6) node [black,midway,yshift=0.2 cm]
{$\scriptstyle 2^{-\alpha_{k+2}}$};

\clip \firstcircle;
\begin{scope}
\draw[fill=gray] (4,0) rectangle (3,1) ;
\draw[red, solid, line width=.5mm] (-4,6)--(-4,1);

\draw[fill=gray] (3,0) rectangle (2,1);

\draw[fill=gray] (2,0) rectangle (1,1);

\draw[fill=gray] (-4,0) rectangle (-3,1);
\draw[red, solid, line width=.5mm] (4,6)--(4,1);

\draw[fill=gray] (-3,0) rectangle (-2,1);
\draw[fill=gray] (-2,0) rectangle (-1,1);

\draw[red, solid, line width=.5mm] (4,1)--(1,1);

\draw[red, solid, line width=.5mm] (-4,1)--(-1,1);

\draw[red, solid, line width=.5mm] (-1,1)--(-1,.25);

\draw[red, solid, line width=.5mm] (1,1)--(1,.25);

\draw[red, solid, line width=.5mm] (-1,.25)--(-.25,.25);

\draw[red, solid, line width=.5mm,] (1,.25)--(.25,.25);
\end{scope}

\begin{scope}[scale=0.5, shift={(-1,0)}]
\draw[fill=gray] (3,0) rectangle (2.5,0.5);
\draw[fill=gray] (2.5,0) rectangle (2,0.5);
\draw[fill=gray] (2,0) rectangle (1.5,0.5);
\end{scope}
\begin{scope}[scale=0.5, shift={(1,0)}]
\draw[fill=gray] (-3,0) rectangle (-2.5,0.5);
\draw[fill=gray] (-2.5,0) rectangle (-2,0.5);
\draw[fill=gray] (-2,0) rectangle (-1.5,0.5);
\end{scope}
\begin{scope}
\draw[fill=gray] (-0.25,0) rectangle (-0.05,0.08);
\draw[red, solid, line width=.3mm,] (-0.25,0.25)--(-0.25,0.08);

\draw[red, solid, line width=.3mm,] (-0.25,0.08)--(-0.05,0.08);

\draw[fill=gray] (0.25,0) rectangle (0.05,0.08);
\draw[red, solid, line width=.3mm,] (0.25,0.25)--(0.25,0.08);

\draw[red, solid, line width=.3mm,] (0.25,0.08)--(0.05,0.08);
\end{scope}

\end{tikzpicture}
\caption{The set $(B_k\cap \Omega)\setminus \Omega_{\star}$.}
\label{corkscrewomegafigure}
\end{figure}
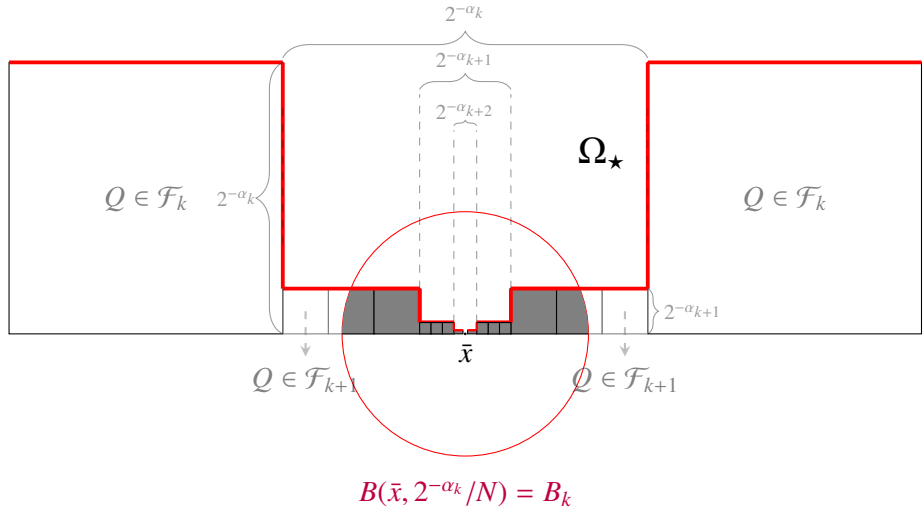

	\def\cprime{$'$} \def\cprime{$'$}

\end{document}